\documentclass[reqno]{amsart}
\usepackage{nomencl}
\makenomenclature
\usepackage[english]{babel}
\usepackage{tikz-cd}
\usepackage{amssymb,latexsym,amscd,amsmath,amsthm,manfnt,tikz}
\usepackage{mathdots}
\usepackage{xcolor}
\usepackage{hyperref}
\hypersetup{colorlinks=true,linkcolor=blue, citecolor=blue}
\usepackage[widespace]{fourier}
\usepackage{frcursive}
\usepackage[pagewise]{lineno}
\usepackage[widespace]{fourier}

\usepackage[matrix,arrow]{xy}
\begin{document}
\input{amssym.def}
\input{xy}

\xyoption{all}
\theoremstyle{remark}
\theoremstyle{plain}
\newtheorem{exam}{Example}
\pagestyle{plain}
\newtheorem*{theorem*}{Theorem}
\setcounter{tocdepth}{1}
\numberwithin{equation}{subsection} 
\newtheorem{guess}{theorem}[section]
\newtheorem{thm}[guess]{Theorem}
\newtheorem{lem}[guess]{Lemma}
\newtheorem{prop}[guess]{Proposition}
\newtheorem{Cor}[guess]{Corollary}
\newtheorem{defi}[guess]{Definition}
\newcommand{\ds}{\displaystyle}
\theoremstyle{definition}
\newtheorem{rem}[guess]{Remark}

\newtheorem{ex}[guess]{Example}
\newtheorem{prob}[guess]{Problem}
\newtheorem{claim}[guess]{Claim}

\newcommand{\Hfr}{\mathfrak{H}}

\newcommand{\fX}{\mathfrak{X}}
\newcommand{\gfr}{\mathfrak{g}}
\newcommand{\hfr}{\mathfrak{h}}
\newcommand{\Dfr}{\mathfrak{D}}
\newcommand{\Efr}{\mathfrak{E}}
\newcommand{\Gfr}{\mathfrak{G}}
\newcommand{\Ffr}{\mathfrak{F}}
\newcommand{\Cfr}{\mathfrak{C}}
\newcommand{\rfr}{\mathfrak{r}}
\newcommand{\lfr}{\mathfrak{l}}
\newcommand{\sfr}{\mathfrak{s}}
\newcommand{\pfr}{\mathfrak{p}}
\newcommand{\Sfr}{\mathfrak{S}}

\newcommand{\cY}{\mathcal{Y}}

\newcommand{\cV}{\mathcal{V}}
\newcommand{\cZ}{\mathcal{Z}}
\newcommand{\cU}{\mathcal{U}}
\newcommand{\cI}{\mathcal{I}}
\newcommand{\cK}{\mathcal{K}}
\newcommand{\cD}{\mathcal{D}}
\newcommand{\cQ}{\mathcal{Q}}
\newcommand{\cO}{\mathcal{O}}
\newcommand{\cC}{\mathcal{C}}
\newcommand{\cE}{\mathcal{E}}
\newcommand{\cG}{\mathcal{G}}
\newcommand{\cW}{\mathcal{W}}
\newcommand{\cB}{\mathcal{B}}
\newcommand{\cF}{\mathcal{F}}
\newcommand{\cH}{\mathcal{H}}
\newcommand{\cM}{\mathcal{M}}
\newcommand{\cA}{\mathcal{A}}
\newcommand{\cN}{\mathcal{N}}
\newcommand{\cR}{\mathcal{R}}
\newcommand{\cL}{\mathcal{L}}
\newcommand{\cT}{\mathcal{T}}
\newcommand{\cS}{\mathcal{S}}
\newcommand{\cX}{\mathcal{X}}
\newcommand{\cP}{\mathcal{P}}
\newcommand{\DD}{\mathcal{D}}
\newcommand{\KK}{\mathcal{K}}
\newcommand{\bs}{\mathbf{s}}
\newcommand{\bb}{\mathbf{b}}
\newcommand{\mf}{\mathfrak{f}}

\newcommand{\Prym}{\mathrm{Prym}}
\newcommand{\Nbd}{\mathrm{Nbd}}
\newcommand{\Frac}{\mathrm{Frac}}
\newcommand{\AbShv}{\mathrm{AbShv}}
\newcommand{\Shv}{\mathrm{Shv}}
\newcommand{\PreShv}{\mathrm{Preshv}}
\newcommand{\Sets}{\mathrm{Sets}}
\newcommand{\topgp}{\mathrm{top}}
\newcommand{\sub}{\mathrm{sub}}
\newcommand{\Sym}{\mathrm{Sym}}
\newcommand{\disc}{\mathrm{disc}}
\newcommand{\Pic}{\mathrm{Pic}}
\newcommand{\Gal}{\mathrm{Gal}}
\newcommand{\Jac}{\mathrm{Jac}}
\newcommand{\Obj}{\mathrm{Obj}}
\newcommand{\Stab}{\mathrm{Stab}}
\newcommand{\Div}{\mathrm{Div}}
\newcommand{\Grass}{\mathrm{Grass}}
\newcommand{\height}{\mathrm{height}}
\newcommand{\Bl}{\mathrm{Bl}}
\newcommand{\Ad}{\mathrm{Ad}}
\newcommand{\Inn}{\mathrm{Inn}}
\newcommand{\Out}{\mathrm{Out}}

\newcommand{\Ram}{\mathrm{Ram}}
\newcommand{\diag}{\mathrm{diag}}
\newcommand{\ram}{\mathrm{ram}}
\newcommand{\Mor}{\mathrm{Mor}}
\newcommand{\Img}{\mathrm{Img}}
\newcommand{\Spec}{\mathrm{Spec}}
\newcommand{\zero}{\mathrm{zero}}
\newcommand{\SU}{\mathrm{SU}}
\newcommand{\Ker}{\mathrm{Ker}}
\newcommand{\Trace}{\mathrm{Trace}}
\newcommand{\ad}{\mathrm{ad}}
\newcommand{\stdn}{\mathrm{stdn}}
\newcommand{\tr}{\mathrm{tr}}
\newcommand{\triv}{\mathrm{triv}}
\newcommand{\sgn}{\mathrm{sgn}}
\newcommand{\ev}{\mathrm{ev}}
\newcommand{\Tor}{\mathrm{Tor}}
\newcommand{\opp}{\mathrm{opp}}
\newcommand{\Cov}{\mathrm{Cov}}
\newcommand{\Covet}{\mathrm{Cov_{\acute{e}t}}}
\newcommand{\ob}{\mathrm{ob}}
\newcommand{\AbGps}{\mathrm{AbGps}}
\newcommand{\rank}{\mathrm{rank}}
\newcommand{\supp}{\mathrm{supp}}
\newcommand{\glue}{\mathrm{\tt glue}}

\newcommand{\lra}{\longrightarrow}
\newcommand{\Lra}{\Leftrightarrow}
\newcommand{\Ra}{\Rightarrow}
\newcommand{\hra}{\hookrightarrow}
\newcommand{\sr}{\stackrel}
\newcommand{\dra}{\dashrightarrow}
\newcommand{\ra}{\rightarrow}
\newcommand{\ol}{\overline}
\newcommand{\wh}{\widehat}
\newcommand{\loms}{\longmapsto}
\newcommand{\la}{\leftarrow}
\newcommand{\lems}{\leftmapsto}
\newcommand{\vp}{\varpi}
\newcommand{\ep}{\epsilon}
\newcommand{\La}{\Lambda}
\newcommand{\abf}{\bigtriangleup_\theta}
\newcommand{\ms}{\mapsto}
\newcommand{\evt}{\tilde{\ev}}
\newcommand{\ul}{\underline}
\newcommand{\uT}{\underline{T}}
\newcommand{\cech}{\cC^.(\fU, \uT)}
\newcommand{\bc}{{\mathbb C}}
\newcommand{\tcr}{\text{\cursive r}}
\newcommand{\bp}{{\mathbb P}}
\newcommand{\bz}{{\mathbb Z}}
\newcommand{\bq}{{\mathbb Q}}
\newcommand{\bn}{{\mathbb N}}
\newcommand{\bg}{{\mathbb G}}
\newcommand{\br}{{\mathbb R}}
\newcommand{{\bh}}{{\mathbb H}}


\newcommand{\vf}{\varphi}

\newcommand{\vt}{\vartheta}

\newcommand{\TT}{\Theta}
\newcommand{\spec}{{\rm Spec}\,}

\newcommand{\gt}{\theta}
\newcommand{\RR}{\mathbb{R}}
\newcommand{\EE}{\mathbb{E}}
\newcommand{\WW}{\mathbb{W}}
\newcommand{\FF}{\mathbb{F}}
\newcommand{\VV}{\mathbb{V}}
\newcommand{\qq}{\mathbb{q}}
\newcommand{\HH}{\mathbb{H}}
\newcommand{\PP}{\mathbb{P}}
\newcommand{\Sbb}{\mathbb{S}}
\newcommand{\ZZ}{\mathbb{Z}}

\newcommand{\QQ}{\mathbb{Q}}
\newcommand{\NN}{\mathbb{N}}
\newcommand{\UU}{\mathbb{U}}
\newcommand{\XX}{\mathbb{X}}
\newcommand{\CC}{\mathbb{C}}
\newcommand{\bX}{\mathbb{X}}
\newcommand{\bY}{\mathbb{Y}}
\newcommand{\fU}{\mathfrak{U}}

\newcommand{\tcm}{\text{\cursive c}}
\newcommand{\tcs}{\small\text{\cursive s}}
\newcommand{\tcq}{\small\text{\cursive q}}
\newcommand{\tcf}{\tiny\text{\cursive f}}
\newcommand{\tcb}{\it b}
\newcommand{\tcge}{\small\text{\cursive ge}}


\newcommand{\sma}[1]{ \begin{pmatrix}

\begin{smallmatrix} #1 \end{smallmatrix} \end{pmatrix} }

\newcommand{\sm}[1]{

\begin{smallmatrix} #1 \end{smallmatrix}  }

\newcommand{\symp}[1]{\mathrm{Sp}_{#1}(\ZZ)}
\renewcommand\qedsymbol{\tt{Q.E.D}}

\newcommand{\Ext}{\mathrm{Ext}}
\newcommand{\std}{\mathrm{std}}
\newcommand{\End}{\mathrm{End}}
\newcommand{\Mum}{\mathrm{Mum}}
\newcommand{\Hom}{\mathrm{Hom}}
\newcommand{\Ind}{\mathrm{Ind}}
\newcommand{\Id}{\mathrm{Id}}
\newcommand{\dimn}{\mathrm{\text{dim}}}
\newcommand{\pardeg}{\mathrm{pardeg}}
\newcommand{\x}{\mathrm{x}}
\newcommand{\Res}{\mathrm{Res}}
\newcommand{\Endq}{\mathrm{End}_\QQ}
\newcommand{\GL}{\mathrm{GL}}
\newcommand{\SL}{\mathrm{SL}}
\newcommand{\Cl}{\mathrm{Cl}}
\newcommand{\SO}{\mathrm{SO}}
\newcommand{\Sp}{\mathrm{Sp}}
\newcommand{\Spin}{\mathrm{Spin}}
\newcommand{\PSO}{\mathrm{PSO}}
\newcommand{\PGL}{\mathrm{PGL}}
\newcommand{\Cb}{\mathrm{Cb}}
\newcommand{\cg}{\mathrm{Cg}}
\newcommand{\Nm}{\mathrm{Nm}}
\newcommand{\Proj}{\mathrm{Proj}}
\newcommand{\im}{\mathrm{im}}
\newcommand{\Aut}{\mathrm{Aut}}
\newcommand{\coker}{\mathrm{coker}}
\def\map#1{\ \smash{\mathop{\longrightarrow}\limits^{#1}}\ }


\newcommand{\Pdon}{\Prym(\pi,\Lambda)}
\newcommand{\twinv}{{H^1(Z,\underline{T})}^W}
\newcommand{\He}{H_{\acute{e}t}}
\newcommand{\Shalf}{S_{1/2}}

\title{Bruhat-Tits group schemes over higher dimensional base-II}

\author[V. Balaji]{Vikraman Balaji}
\address{Chennai Mathematical Institute }
\email{balaji@cmi.ac.in}

\author[Y. Pandey]{Yashonidhi Pandey}
\thanks{The research for this paper was partially funded by the SERB Core Research Grant CRG/2022/000051.}
\address{ 
Indian Institute of Science Education and Research, Mohali Knowledge city, Sector 81, SAS Nagar, Manauli PO 140306, India}
\email{ ypandey@iisermohali.ac.in, yashonidhipandey@yahoo.co.uk, pandeyyashonidhi@gmail.com}

\begin{abstract}
We prove  that  split reductive BT group schemes over a higher dimensional base  are {\em affine}. Our method also gives a new construction of higher BT-group schemes more general than parahoric ones. The new ingredients are an extension of J.-K.Yu's construction in \cite{yu} to higher dimensional bases,  N\'eron-Raynaud dilatations of subgroup schemes on divisors, combined with techniques from \cite{bt2} and the structure theory developed in \cite{bp}.

\end{abstract}
\maketitle

\small
\tableofcontents
\setcounter{tocdepth}{2} 
\normalsize

\section{Introduction} \label{intro} 

Let $K$ be a local field complete with respect to a valuation which is not necessarily discrete. Let $G$ be a  be a connected {\it quasi-split} reductive algebraic group over $K$. A bounded subgroup ${\tt Q} \subset G(K)$ is called  {\sl schematic} if there exists a smooth $\cO$-group scheme $\mathfrak{G}$, with connected fibres, of finite type, with $G_{K}$ as generic fibre and such that ${\tt Q} = \mathfrak{G}(\cO)$.
In this setting, 
 Bruhat-Tits theory provides integral models of certain bounded subgroups of $G(K)$. Let us introduce some notions to make this more precise.
 
Let $\boldsymbol{\Phi}$ denote the root system of $G$ relative to a maximal split torus $T$.  Let $\cA_{T}:=\Hom(\mathbb{G}_m,T) \otimes \mathbb{R}$ denote the apartment of $T$. For a point $\theta \in \cA_{T}$ and a root $r \in \Phi$, we set 
\begin{equation} \label{mrtheta} 
m_{r} \left(\theta \right) := - \left\lfloor{(r, \theta)} \right\rfloor.
\end{equation} 
Let $\cO$ be a complete discrete valuation ring with residue field $k$, with uniformizer $z$ and field of fractions $K$. Classically (see \cite{bruhattits1}), in terms of generators a parahoric subgroup ${\cP}_{\theta}$ of $G(K)$ is defined as
\begin{equation} \label{parahoricwithgen}
{\cP}_{\theta}:= \left\langle T(\cO), U_r\left( z^{m_r(\theta)} \cO\right), r \in \boldsymbol{\Phi} \right\rangle.
\end{equation}
 Maximal bounded subgroups of $G(K)$ are schematic and in fact parahoric. More generally,  for a bounded subset $\Omega$ of $\cA_T$ we set 
\begin{equation}\label{mromega}
  m_r(\Omega): = -\left\lfloor \inf_{\theta \in \Omega} r(\theta) \right\rfloor = \left\lceil \sup_{\theta \in \Omega} -r(\theta) \right\rceil. 
\end{equation}

Let $\tilde{\boldsymbol{\Phi}}:=\boldsymbol{\Phi} \cup \{0 \}$.
Recall, see \cite[Section~6.4.3, p.~133]{bruhattits1},
that a function $f\colon  \tilde{\boldsymbol{\Phi}} \rightarrow \mathbb{R}$ is said to be {\em concave} if whenever $r_{i} \in \tilde{\boldsymbol{\Phi}}$ are such that $\sum_i r_{i} \in \tilde{\boldsymbol{\Phi}}$, we have
\begin{equation}\label{concave}
f \left( \sum r_{i} \right) \leq \sum f(r_{i}).
\end{equation}
In \cite{bp} we say that a concave function $f$ is parahoric or equivalently of  type I, if there is a point $\theta$ in the apartment $\mathcal{A}_{_T}$ of $T$ such that for all $r \in \tilde{\boldsymbol{\Phi}}$ we have 
$$f(r)=\lceil -r(\theta) \rceil.$$

The functions $r \mapsto m_r(\Omega)$ are called concave functions of type II  and form a strictly smaller subclass of the set of all concave functions (see \cite[Def 1.2]{bp} for the details; see also \cite[\S 2.0.1]{insa}).

To each concave function $f$, one can associate unique upto a $G(K)$-conjugation, a bounded subgroup of $G(K)$ as follows:
\begin{equation}
{\cP}_{f}:= \left\langle T(\cO),U_r\left(  z^{\lceil f(r) \rceil} \cO \right), r \in \Phi \right\rangle.
\end{equation}
For each concave function $f$, Bruhat-Tits theory \cite{bt2} construct a smooth group scheme $\mathfrak{G}_{f} \rightarrow \textrm{Spec}(\mathcal{O})$ with generic fiber $G$ such that as subgroups of $G(K)$ we have
$$\mathfrak{G}_{f}(\mathcal{O}) = {\cP}_{f}.$$
 
 The main method employed in \cite{bt2} consists of taking schematic closures in faithful $\mathcal{D}$-modules. Later, over discrete valuation rings, J.-K.Yu gave a different construction of Bruhat-Tits group schemes $\mathfrak{G}_{f}$  \cite{yu}  by a process involving dilatations, starting with parahoric ones. A somewhat different geometrical approach for parahoric group schemes was observed by Balaji-Seshadri \cite{base}. They realised parahoric group schemes as invariant direct images, or equivalently {\em Weil restriction of scalars plus invariants}, over arbitrary complete discrete valuation rings. Further they proved the following global result over a smooth projective curve $X$.  A parahoric Bruhat-Tits group scheme $ \mathfrak{G}$ over a smooth projective curve $X$ is one whose generic fibre $\mathfrak{G}_{k(X)}$ is an absolutely simple, simply connected group scheme, which splits over a Galois extension and whose restriction to the formal neighbourhood of a finite set of closed points $x \in X$ are given by the usual parahoric group schemes. One of the main theorems of \cite{base} is the following.
 
   \begin{thm} (see \cite[Theorem 5.2.7, Remark 5.2.8]{base}) \label{mtbase}  For any parahoric Bruhat-Tits group scheme $\mathfrak{G}_X$, there exists a Galois cover $\psi:Y \to X$ with Galois group $\Gamma$, and a {\em semisimple group scheme} $\mathfrak{G}'_Y$ on $Y$, with split generic fibre $\mathfrak{G}'_{k(Y)}$ isomorphic to $G \otimes_{_{k(X)}} k(Y)$  so that $\mathfrak{G}_X$  is isomorphic to the {\em invariant direct image} of $\mathfrak{G}'_Y$, i.e.  
\begin{equation}\label{papparappakabaap}
\mathfrak{G}_X \simeq \psi_{_*}^{^{\Gamma}}(\mathfrak{G}'_Y) := \text{\tt Res}_{_{Y/X}}^{^{\Gamma}}(\mathfrak{G}'_Y).
\end{equation} 
\end{thm}
See also \cite[Thm 1]{pr2024} for this result in positive characteristics under tameness assumptions of the residue field $k$. 

The question of generalization of Bruhat-Tits group schemes to bases of higher dimensions was raised and outlined by F.Bruhat and J.Tits themselves in \cite[ \S 3.9.4]{bt2}. The possible difficulties were also carefully elaborated in a few `possible' counter-examples for $\mathbb{A}^2_k$ and $SL_2$. In \S 3.2.14 {\em loc. cit.}, it is shown that in higher dimensions, taking schematic closures in different faithful $\mathcal{D}$-modules may produce non-isomorphic group schemes $\mathfrak{G}$  having the same structure of the big-cell. In \S 3.2.15 {\em loc. cit.}, it is shown that the schematic closure $\mathfrak{G}$ may {\em not be flat} and hence {\em not  a group scheme}. By removing a four dimensional component $\mathfrak{G}^{^1}_{_{(0,0)}}$,  although one gets a group scheme $\mathfrak{H}:= \mathfrak{G} - \mathfrak{G}^{^1}_{_{(0,0)}} $, but the important question of {\em its affineness is left in doubt}. 


The aim of this work is to prove, under some mild and natural  assumptions on the characteristic of the residue field, that  these higher {\tt BT}-group schemes exist and  are always smooth and {\em affine}; in fact we prove this under more general hypothesis, namely when $G$ is  {\em split reductive}, a hypothesis forced on us due to the inductive argument. We refer the reader to \cite[\S 3.2]{bt2} for the relevance of the split case. As an added benefit, we also obtain a new construction of the $n$-{\tt BT}-group schemes when the concave functions are more general than of type I (called of types II and III in {\cite{bp}) under milder hypothesis than those of {\em loc. cit.}, the proof relying nevertheless on  the existence of type I group schemes. This new approach also reveals certain finer aspects of the structure of these group schemes. It is based on the inductive process developed by J.-K.Yu \cite{yu} combined with the methods of \cite{bt2} and of course relying entirely on the earlier work \cite{bp}. We now state our main theorem.

\begin{thm} \label{mt}  
Let $X$ be a smooth quasi-projective scheme over $k$, and let  
 \begin{equation} D = \sum_{_{i=0}}^{^n} D_{_i},
 \end{equation}
be a reduced normal crossing divisor. For $0 \leq i \leq n$, we denote the generic point of the component $D_{_i}$ as 
\begin{equation} \zeta_{_i}.
\end{equation} Let us set 
\begin{eqnarray} A_{_i} & := & \mathcal O_{_{X, \zeta_{_{i}}}}, \\
 X_{_i} & := & \spec(\hat{A}_{_i}), \\
X_{_o} & := & X - D.
\end{eqnarray}
Let $G$ be a split reductive connected Chevalley group scheme over $\mathbb{Z}$ with a split maximal torus $T$.  Let $\boldsymbol{\boldsymbol{\Phi}}$ denote the root system of $G$ relative to $T$. For $0 \leq i \leq n$, let $f_{_i}: \boldsymbol{\Phi} \cup \{0 \} \rightarrow \mathbb{R}$ be  concave functions with $f_{_i}(0)=0$. Let $\mathfrak G_{_{f{_{_i}}}} \rightarrow X_{_i}$ be the {\tt BT} group scheme on $X_{_i}$ associated to the concave function $f_{_i}$. For any $\left\{t_{_i} \in T \left(X_{_o} \times_X X_{_i} \right) \right\}$, we say that there is a smooth group scheme $\mathfrak G$ on $X$  with connected fibres, unique upto isomorphism which interpolates the given datum $\left\{(D_{_i},f_{_i},t_{_i}) | 0 \leq i \leq n \right\}$ if
\begin{enumerate}
\item $\Gfr|_{_{X_{_o}}} \simeq G \times X_{_o}$,
\item $\Gfr|_{_{X_{_i}}} \simeq \Gfr_{_{f{_{_i}}}}$, for each $i$,
\item glued by $t_i$ on $X_{_o} \times_X X_{_i}$,
\item the big-cell structure on $\mathfrak{G}$ extends those on $G \times X_{_o}$ and $\Gfr_{_{f{_{_i}}}}$.
\end{enumerate}

Let $\theta_{_i} \in \text{Hom}(\mathbb{G}_{_m},T) \otimes \mathbb{Q}$ be rational points such that we have a morphism of BT-group schemes $\mathfrak{G}_{_{f_{_i}}} \ra \mathfrak{G}_{_{\theta_{_i}}}$ inducing identity at the generic fiber. Let positive integers $d_{_i}$ be least such that $d_{_i} \theta_{_i}$ belongs to $\text{Hom}(\mathbb{G}_{_m}, T)$. Let $k$ be a perfect field of characteristic coprime to all the $d_{_i}$ and containing the $d_{_i}$-th roots of unity. 

Then upto isomorphism there is a unique smooth {\em affine} group scheme $\mathfrak G$ on $X$  with connected fibres  which interpolates the given datum $\left\{(D_{_i},f_{_i},t_{_i}) | 0 \leq i \leq n \right\}$.
\end{thm} 

\noindent


We refer the reader to the expository article \cite{insa} for  some of the contexts of our present work. The framework of our main result is also that of \cite{bp} except that there we have assumed that $G$ is semisimple. The special case of Theorem \ref{mt}, when each $f_{_i}$ is parahoric, is essentially \cite[Theorem 5.7, \S 5.9]{bp}. 

Over regular base rings of dimension two,  group schemes were constructed in \cite[Section 3.0.9]{bt2}. But they were not proven to be affine. In dimension two, for a tuple of the type $\{0,f\}$ of concave functions,  we refer the reader to Pappas and Zhu \cite[Section 4]{pz} for the affineness of these group schemes.

Building on the construction of $n$-parahoric group schemes, in \cite[\S 6, \S 7]{bp}, we constructed smooth $n$-{\tt BT}-group schemes interpolating the datum of $n$-tuples of concave functions.  However, our methods there using \cite[Proposition 2.2.10]{bt2} could show  them to be  only {\em quasi-affine}. As is well-known, it is important for applications to know whether these are in fact {\em affine}. 

\subsubsection{Layout} The assumption on $G$ in \cite{bp} is that it is {\em split, semi-simple, and simply connected}. We first show that \cite[Theorem 5.5]{bp} holds  for $n$-parahoric group scheme when the generic group is assumed to be more generally, {\em split connected reductive}.  Under these general hypotheses  we first prove the existence and affineness of the $n$-{\tt BT}-group schemes when  the base X has {\em dimension two}. These two results provide the inductive datum for the proof of the existence and affineness of the $n$-{\tt BT}-group schemes when the base $X$ is higher dimensional.

\section{The recursive step of J.K. Yu \cite{yu}}
Although the results in this article are independent of \cite{yu}, the broad contour of concave functions in J.K. Yu's article is the foundation for the proof strategy of this work.

\subsection{Over arbitrary DVRs} \label{advr} In \cite{yu}, $\cO$ is a {\it Henselian} dvr with {\it perfect} residue field $\kappa$. However, in this paragraph we outline his recursive structure, without imposing the {\em perfectness} of the residue field and the {\em henselian} property of the dvr. But we assume that $G$ is a split Chevalley group scheme with maximal torus $T$ and $\cO$ is an arbitrary DVR. Let $f: \boldsymbol{\Phi} \cup \{0\} \ra \mathbb{R}$ be a concave function with $f(0)=0$ and let $\Gfr_{f}$ be the BT-group scheme on $\spec(\cO)$ defined by $f$. Thus, $\mathfrak{G}_{_{f}}$ contains the toral group scheme $$\mathfrak{T}:= T \times_{_\mathbb{Z}} \Spec(\cO).$$ Let $\Gfr_{_\theta}(\cO)$ be any maximal parahoric subgroup of $G(K)$ which contains  $\Gfr_{_f}(\cO)$, where $\theta$ is a point in the affine apartment $\cA_T$. Let 
$f_{_\theta}$ denote the linear concave function which on $r \in \boldsymbol{\Phi} \cup \{0\}$ gives $-r(\theta)$. 
Then we have 
$$f_{_\theta} \leq f.$$
For any $t \in [0,1]$ consider, the concave function
$$f_t:= (1-t) f_{_\theta} + t f.$$ 
Let $\lceil f \rceil$ denote the concave function such that $\lceil f \rceil (r)= \lceil f(r) \rceil$.
As $t$ increases, we can look at "small jumps", i.e., $0\leq s<e \leq 1$ such that $\lceil f_e \rceil \neq \lceil f_s \rceil$ however 
\begin{equation}\label{fefs} f_e(r) \leq f_s(r)+1 \quad \text{for all} \quad r \in \boldsymbol{\Phi}.\end{equation}
These jumps give connected Zariski closed algebraic subgroups in the closed fiber of $\Gfr_{_{f_s}}$:
\begin{equation} \label{Hkappa} H_{_{\kappa}} := \text{Image} \left(\Gfr_{_{f_e}} \ra \Gfr_{_{f_s}} \right)|_{\kappa}.
\end{equation}

As a subgroup of the closed fiber $\Gfr_{_{f_s}}|_{\kappa}$ this group is generated by the images of the following: the maximal torus and the collection of subgroups   $U_{r,e}|_{\kappa}$ of unipotent subgroups $U_{r,e} \subset \Gfr_{_{f_e}}$ at the closed fiber for roots $r \in \boldsymbol{\Phi}$ satisfying
\begin{equation} \label{rootsH} \lceil f_s(r) \rceil = \lceil f_e(r) \rceil,
\end{equation} (see \cite[paragraph before 7.3.1]{yu}). 

The recursive step of J.-K. Yu (see \cite[paragraphs before Lemma 8.1.2 and Lemma of \S 9.2]{yu}) reverses this construction. It starts with $\Gfr_{_{f_s}}$ and using $f_e$ defines $H_{_{\kappa}}$ to construct $\Gfr_{_{f_e}}$ by N\'eron-Raynaud dilatation of $\Gfr_{_{f_s}}$ along $H_{_{\kappa}}$. 

We recall the following notation and definition. Let $\mathcal G_{_{\cO}}$ be a smooth group scheme over $\cO$. Let $Z_{_k} \subset \mathcal G_{_{k}}$ be a closed subgroup. Then we denote by 
\begin{equation}\label{nerondilat}\mathcal G^{^{Z_{_{k}}}}_{_{\cO}}\end{equation} the  {\bf N\'eron-Raynaud dilatation} of  $\mathcal G_{_{\cO}}$ along $Z_{_k}$. We will simply call this a `dilatation' in this article. In this work, we use a higher dimensional analogue of dilatations which appears in \cite{mayeux}. 

Thus, starting with {\em parahoric group schemes} over discrete valuation rings,  and  taking iterated dilatations along subgroups $H_{\kappa}$ of the closed fibre of the group scheme as above, one can construct BT-group schemes associated to any concave function. We remind the reader  that BT-group schemes associated to concave functions of type I are precisely the parahoric group schemes.

\subsubsection{On a result of McNinch on Levi decompositions} Let $\mathcal{O}$ be a complete discrete valuation ring with quotient field $K$ and residue field ${{\kappa}}$, {\em not assumed to be perfect}. Let $\mathcal{P} \ra \Spec(\mathcal{O})$ be a connected parahoric group scheme containing a connected toral group scheme $\mathfrak{T} \hookrightarrow \cP$. Assume that $\mathfrak{T}_{_{K}}$ is maximal $K$-split. By \cite[Thm 1]{mcn} if $G_{_{K}} = \cP_{_{K}}$ splits over an unramified extension of $K$, then there is a reductive group scheme $\mathcal{M} \ra \Spec(A)$ containing $\mathfrak{T}$ such that $\cP_{_{{{\kappa}}}}$ has a Levi-decomposition with Levi factor $\cM_{_{{{\kappa}}}}$. The following generalises this result to arbitrary BT-group schemes under the assumption that $G$ is a connected split Chevalley group scheme.

 \begin{thm} \label{mcninch} Let the setting be as in \S \eqref{advr}, except that we do not assume the residue field is perfect.
 We have
 \begin{enumerate}
 \item[a)] The canonical quotient morphism $\mathfrak{G}_{_{f_{_{e,{\kappa}}} }} \ra H_{_{\kappa}}$ induces a natural isomorphism of the reductive quotients $$\mathfrak{G}^{^{\tt red}}_{_{f_{_{e,{\kappa}}}}}= H^{^{\tt red}}_{_{{\kappa}}}.$$
 \item[b)]  Fixing the inclusion $\mathfrak{T}_{_{{\kappa}}} \hookrightarrow \mathfrak{G}_{_{f_{_{e,{\kappa}}}}}$, the group $H_{_{{\kappa}}}$ has a natural Levi-decomposition.  
 \item[c)] For any concave function $f: \boldsymbol{\Phi} \cup \{ 0 \} \ra \mathbb{R}$ the BT-group scheme $\mathfrak{G}_{_{f}} \ra \Spec(A)$ has a reductive subgroup scheme $\cM$ containing $\mathfrak{T}$ such that $\mathfrak{G}_{_{f,{\kappa}}}$ has a Levi-decomposition with Levi factor $\cM_{_{{\kappa}}}$.
 \end{enumerate}
 \end{thm}
 \begin{proof} The proof assume the results of McNinch as the starting hypothesis. In particular, we assume the theorem for the group scheme $\Gfr_{_{f_s}}$. We abbreviate the $\mathcal{O}$-group schemes $\Gfr_{_{f_s}}$ by $\mathfrak{G}_{_{\mathcal{O}}}$ and $\Gfr_{_{f_e}}$ by $\Gfr'_{_{\mathcal{O}}}$.  We have a natural factorisation
$$\Gfr'_{_{\kappa}} \twoheadrightarrow H_{_{\kappa}} \hookrightarrow \mathfrak{G}_{_{{\kappa}}},$$
where the surjection above restricts to an isomorphism on $\mathfrak{T}_{_{\kappa}}$. As a subgroup of $\mathfrak{G}_{_{\kappa}}$, $H_{_{\kappa}}$ is generated by 
$$\left\langle \mathfrak{T}_{_{\kappa}}, U_{_{a,{\kappa}}} \quad \text{where} \quad f_{_{e}}(a)=f_{_{s}}(a) \quad \text{for} \quad a \in \boldsymbol{\Phi} \right\rangle.$$

With respect to $\mathfrak{T}_{_{\kappa}}$, let $\boldsymbol{\Phi}_{_{s}}$ (resp. $\boldsymbol{\Phi}_{_{e}}$) denote the root system of $\mathfrak{G}_{_{f_{_{s}}, \kappa}}$ (resp. $\mathfrak{G}_{_{f_{_{e}}, {\kappa}}}$).
 Thus, the root system of the reductive reduction can be described as follows :
\begin{equation}
\boldsymbol{\Phi}^{^{\tt red}}_{_{f_{_{e}}}} = \left\{ a \in \boldsymbol{\Phi} | f_{_{e}}(a)+ f_{_{e}}(-a) =0 \right\},
\end{equation} 
(see \cite[Corollaire 4.6.12 (i)]{bt2}). By \eqref{Hkappa}, it follows that for $a \in \boldsymbol{\Phi}_{_{e}}~\setminus~\boldsymbol{\Phi}^{^{\tt red}}_{_{f_{_{e}}}}$ the image of $U_{_{a,\kappa}}$ in $H_{_{\kappa}}$ vanishes. On the other hand, since for all $a \in \boldsymbol{\Phi}$, $f_{_{s}}(a)+f_{_{s}}(-a) \geq 0$ we see that
$$f_{_{e}}(a) > f_{_{s}}(a) \implies a \in \boldsymbol{\Phi}_{_{e}} \setminus \boldsymbol{\Phi}^{^{\tt red}}_{_{f_{_{e}}}},$$
i.e.,  $a \in \boldsymbol{\Phi}^{^{\tt red}}_{_{f_{_{e}}}} \implies f_{_{e}}(a) = f_{_{s}}(a)$.  Hence by \eqref{rootsH}, it follows that 
$$H^{^{\tt red}}_{_{{\kappa}}}=\left\langle  \mathfrak{T}_{_{\kappa}}, U_{_{a,\kappa}} | a \in \boldsymbol{\Phi}^{^{\tt red}}_{_{f_{_{e}}}} \right\rangle,$$
which is $\Gfr'^{^{\tt red}}_{_{{\kappa}}}$. Summarising the above, we have
\begin{eqnarray}
\Gfr'^{^{\tt red}}_{_{{\kappa}}} & = & H^{^{\tt red}}_{_{{\kappa}}}, \\
H^{^{\tt red}}_{_{{\kappa}}} & = & \text{Image} \left(\Gfr'^{^{\tt red}}_{_{{\kappa}}} \hookrightarrow \mathfrak{G}^{^{\tt red}}_{_{{\kappa}}} \right).
\end{eqnarray}
This proves a). \\

\noindent
{\ul{We now prove b).}}  By the hypothesis on $\Gfr_{_{f_s}}$, the natural exact sequence on its closed fibre
\begin{equation}\label{levieqn}1 \ra \mathfrak{G}^{^{u}}_{_{{\kappa}}} \ra \mathfrak{G}_{_{{\kappa}}} \stackrel{\pi}{\ra} \mathfrak{G}^{^{\tt red}}_{_{{\kappa}}} \ra 1\end{equation} splits. Let us fix the Levi splitting 
 
  $$\lambda: \mathfrak{G}^{^{\tt red}}_{_{{\kappa}}} \ra \mathfrak{G}_{_{{\kappa}}}.$$ The splitting becomes unique since we fix the inclusion $\mathfrak{T}_{_{{\kappa}}} \hookrightarrow \mathfrak{G}_{_{{\kappa}}}$.  Let $$i: H_{_{{\kappa}}} \ra \mathfrak{G}_{_{{\kappa}}}$$ denote the inclusion morphism.  Similarly the kernel $H^{^{u}}_{_{{\kappa}}}$ of the natural quotient map $$\pi_{_{H}}: H_{_{{\kappa}}} \ra H^{^{\tt red}}_{_{{\kappa}}}$$ is isomorphic to $\mathfrak{G}^{^{u}}_{_{{\kappa}}} \cap H_{_{{\kappa}}}$.
 Thus, we have
 \begin{equation} \label{u-red}
 \xymatrix{
 1 \ar[r] & \mathfrak{G}^{^{u}}_{_{{\kappa}}}  \ar[r] & \mathfrak{G}_{_{{\kappa}}} \ar[r]^{\pi} & \mathfrak{G}^{^{\tt red}}_{_{{\kappa}}} \ar[r] & 1 \\
 1 \ar[r] & H^{^{u}}_{_{{\kappa}}}  \ar@{^{(}->}[u] \ar[r] & H_{_{{\kappa}}} \ar[r]^{\pi_{_H}} \ar@{^{(}->}[u]^{i} & H^{^{\tt red}}_{_{{\kappa}}} \ar[r] \ar@{^{(}->}[u] & 1 
} 
 \end{equation}
 
  Let $\boldsymbol{\Phi}_{_{H}}$ (resp. $\boldsymbol{\Phi}^{^{\tt red}}_{_{H}}$) denote the set of roots of $H_{_{{\kappa}}}$ (resp. the root system of $H^{^{\tt red}}_{_{{\kappa}}}$) with respect to $\mathfrak{T}_{_{{\kappa}}}$. Then we may view the root system $\boldsymbol{\Phi}^{^{\tt red}}_{_{H}}$ as a subset of $\boldsymbol{\Phi}_{_{H}}$. Viewing each $r \in \boldsymbol{\Phi}^{^{\tt red}}_{_{H}}$ as a root in $\boldsymbol{\Phi}_{_{H}}$, we denote the corresponding root homomorphisms by $$u_{_{r}}: \mathbb{G}_{_{a}} \ra H_{_{{\kappa}}}. $$  Consider the subgroup generated by $\mathfrak{T}_{_{{\kappa}}}$ and the images $U_{_{r}}$ of the $u_{_{r}}$, i.e.,
 $$ S_{_{{\kappa}}}:= \left\langle \mathfrak{T}_{_{{\kappa}}}, U_{_{r}}, r \in \boldsymbol{\Phi}^{^{\tt red}}_{_{H}} \right\rangle < H_{_{{\kappa}}}.$$
 
 Since we have $$\lambda \circ \pi \circ i \circ u_{_{r}} = i \circ u_{_{r}},$$
 therefore as  subgroups of $\mathfrak{G}_{_{{\kappa}}}$, we have
 \begin{eqnarray*} i \left(S_{_{{\kappa}}} \right)= \left\langle i \left(\mathfrak{T}_{_{{\kappa}}} \right), i \circ u_{_{r}}, r \in \boldsymbol{\Phi}^{^{\tt red}}_{_{H}} \right\rangle = \left\langle \lambda \circ \pi \left(\mathfrak{T}_{_{{\kappa}}} \right),  \lambda \circ \pi \circ i \circ u_{_{r}}, r \in \boldsymbol{\Phi}^{^{\tt red}}_{_{H}} \right\rangle \\ = \lambda \left(\left\langle \pi \left(\mathfrak{T}_{_{{\kappa}}} \right),   \pi \circ i \circ u_{_{r}}, r \in \boldsymbol{\Phi}^{^{\tt red}}_{_{H}} \right\rangle \right) = \lambda \left(H^{^{\tt red}}_{_{{\kappa}}} \right).
 \end{eqnarray*}
 
 Consequently, by \eqref{u-red} the intersection
 $$i \left(S_{_{{\kappa}}} \right) \cap \mathfrak{G}^{^{u}}_{_{{\kappa}}}$$
 is trivial and hence $S_{_{{\kappa}}} \cap H^{^{u}}_{_{{\kappa}}}$ is trivial. Hence $H_{_{{\kappa}}}$ acquires a Levi-decomposition by the subgroup $S_{_{{\kappa}}}$. This completes the proof of b).\\

\noindent
{\ul{We now prove c).}} 
 As before, we assume the statement c) for $f$ equalling $f_s$ and prove it for $f_e$.

We begin by noting that $\lambda \left(H^{^{\tt red}}_{_{{\kappa}}} \right)$ is a {\em maximal rank reductive subgroup} of the reductive group $\mathfrak{G}^{^{\tt red}}_{_{{\kappa}}}$. {\em From now on by $H^{^{\tt red}}_{_{{\kappa}}}$ we mean the subgroup $\lambda \left(H^{^{\tt red}}_{_{{\kappa}}} \right)$}.  Hence, by the Borel-de Seibenthal Theorem, we can realise this subgroup as the centralizer of its own centre in the bigger group (with the proviso that we avoid small prime characteristics!), i.e., $$H^{^{\tt red}}_{_{{\kappa}}} =  C_{_{\mathfrak{G}^{^{\tt red}}_{_{{\kappa}}}}} \left(Z \left(H^{^{\tt red}}_{_{{\kappa}}} \right)\right)^{^o}.$$

Let us denote the center of $H^{^{\tt red}}_{_{{\kappa}}}$ simply by:
\begin{equation}
Z_{_{{\kappa}}}:= Z\left( H^{^{\tt red}}_{_{{\kappa}}} \right).
\end{equation} 

Thus, $Z_{_{{\kappa}}}$ is a subgroup scheme of $\mathfrak{T}_{_{{\kappa}}}$. These are multiplicative type group schemes. Since $\mathfrak{T}_{_{{\kappa}}}$ spreads to ${\mathcal{O}}$ as a split maximal torus $\mathfrak{T}_{_{{\mathcal{O}}}}$, the subgroup $Z_{_{{\kappa}}}$ spreads as well. We denote the $\mathcal{O}$-lift as 
\begin{equation}
\mathfrak{Z}_{_{{\mathcal{O}}}} \hookrightarrow \mathfrak{T}_{_{{\mathcal{O}}}}.
\end{equation}

By the inductive assumption that c) holds for $f$, the reductive quotient $\mathfrak{G}_{_{{\kappa},red}}$ spreads to  
$$\mathfrak{G}^{^{\tt red}}_{_{{\mathcal{O}}}} \hookrightarrow \mathfrak{G}_{_{\mathcal{O}}}$$
as a reductive subgroup scheme.

Taking centralizers, we now set
\begin{equation}
{H}^{^{\tt red}}_{_{{\mathcal{O}}}} :=C_{_{\mathfrak{G}^{^{\tt red}}_{_{{\mathcal{O}}}}}} \left(\mathfrak{Z}_{_{{\mathcal{O}}}} \right)^{^o},
\end{equation}
which provides a {\em (connected) reductive spread} of ${H}^{^{\tt red}}_{_{{\kappa}}}$.
 
Since $\Gfr_{_{\cO}}$ is a smooth affine ${\mathcal{O}}$-group scheme and $\mathfrak{Z}_{_{{\mathcal{O}}}}$ is a subgroup scheme of multiplicative type, by Theorem\ref{sixthree} we get the inclusions
\begin{equation}\label{firstincl}
{H}^{^{\tt red}}_{_{{\mathcal{O}}}} \hookrightarrow \mathfrak{G}^{^{\tt red}}_{_{{\mathcal{O}}}} \hookrightarrow \mathfrak{G}_{_{\mathcal{O}}}.
\end{equation} 
Thus the reductive spread ${H}^{^{\tt red}}_{_{{\mathcal{O}}}}$ is a {\em {\ul{connected, closed}} and {\ul{smooth}} subgroup scheme}. 
By the universal property of taking dilatations along closed subschemes of the closed fiber, we have the tautological identification:
$$ {H}^{^{\tt red}}_{_{{\mathcal{O}}}} = \left({H}^{^{\tt red}}_{_{{\mathcal{O}}}} \right)^{^{H^{^{\tt red}}_{_{{\kappa}}}}} $$ and also get  the canonical inclusion of \ul{connected} $\mathcal{O}$ group schemes:
$${H}^{^{\tt red}}_{_{{\mathcal{O}}}} \hookrightarrow \mathfrak{G}_{_{\mathcal{O}}}^{^{H^{^{\tt red}}_{_{{\kappa}}}}}.$$
Similarly have the following morphism of \ul{connected} $\mathcal{O}$ group schemes:
$$ \mathfrak{G}_{_{\mathcal{O}}}^{^{H_{_{{\kappa}}}^{^{\tt red}}}} \ra \mathfrak{G}_{_{\mathcal{O}}}^{^{H_{_{{\kappa}}}}}= \mathfrak{G}_{_{f_{_{e}}}} = \Gfr'_{_{\mathcal{O}}}.$$
Composing the above morphisms, it follows that ${H}^{^{\tt red}}_{_{{\mathcal{O}}}}$ is a connected reductive subgroup scheme of $\mathfrak{G}_{_{f_{_{e}}}} = \Gfr'_{_{\mathcal{O}}}$,  giving the sought Levi-decomposition of $\mathfrak{G}_{_{f_{_{e}}}}$ in the closed fiber together with an $\mathcal{O}$-spread.
 \end{proof}

\section{Parahoric Higher BT group scheme in the split reductive case are affine} \label{par-are-affine}

We consider the case when the concave functions $f_{_i}$ are given by a tuple of points $\left(\theta_0, \ldots, \theta_n \right)$ of an affine apartment. Note that in this case we have $f_{_i}(0)=0$. 
 When $G$ is semisimple simply-connected by \cite[Theorem 5.4, Proposition 5.5, Theorem 5.7]{bp}, the smooth group scheme interpolating the given generic datum is shown to be {\it affine, with connected fibers} and to possess a big-cell structure. In what follows, we extend these results to the case $G$ is split reductive.

\subsubsection{When $G$ is not simply-connected} Let $G_{_{sc}}$ denote its simply-connected cover and let $Z:= ker \left(G_{_{sc}} \ra G \right)$ be the finite kernel. Let $T$ denote the maximal torus of $G$ and $T_{_{sc}}$ that of $G_{_{sc}}$. We may identify the apartments $\cA_{_{T}}$ and that of $T_{_{sc}}$. Lifting the gluing functions $t_i$ to $T_{_{sc}}$ and then viewing the tuple of points $(\theta_0, \ldots, \theta_n)$ as those of an affine apartment of $T_{_{sc}}$, by the previous step, we get an affine group scheme $\mathfrak{G}_{_{sc}}$ on $X$.  The subgroup schemes $Z \times X_{_i} \hookrightarrow T \times X_{_i}$ of $\Gfr|_{_{X_{_i}}}$ and $Z \times X_{_o} \hookrightarrow T \times X_{_o}$ of $\Gfr|_{_{X_{_o}}}$ glue by $t_{_i}$ to extend as the product group schemes $$\mathfrak{Z}:=Z \times X \hookrightarrow \mathfrak{T} := T \times X.$$   By \cite[Expose VI, Remarque 9.3, Expose V,7.1]{sga3},  $\mathfrak{G}_{_{sc}}/\mathfrak{Z}$ is representable and is the sought affine prolongation of $G$ to $X$ with connected fibers and with the big-cell given by the image of the big-cell of  $\mathfrak{G}_{_{sc}}$.

\subsubsection{We now consider the case $G$ is split reductive} As before $T$ extends to $X$ as $\mathfrak{T}$ which is in fact the product $T \times X$. Let $G_{_{\tt{der}}}$ be the {\em derived subgroup} of $G$ and $Z_{_{G}}^0$ denote the connected component of identity in its centre. Set $Z:= G_{_{\tt{der}}} \cap Z_{_G}^0$. The apartment of $T$ is simply a product of the apartment of $T_{_{\tt{der}}}:=T \cap G_{_{\tt{der}}}$ with that of $Z_{_{G}}^0$. We project the given tuple of points to the apartment of $T_{_{\tt{der}}}$ and take lifts in $T_{_{\tt{der}}}$ of the images of the gluing functions $t_i$ in $T_{_{ad}}$. Let $\mathfrak{G}_{_{\tt{der}}}$ denote the affine group scheme on $X$ interpolating these. Further $Z \hookrightarrow Z_{_G}^0 \hookrightarrow T$ extend to $X$ as simply the product group schemes $$Z_{_X}:= Z \times X \hookrightarrow Z_{_{G,X}}^0 := Z_{_G}^0 \times X \hookrightarrow \mathfrak{T}.$$Thus, again by {\em loc.cit}, $\left(\mathfrak{G}_{_{\tt{der}}} \times Z_{_{G,X}}^0  \right)/Z_X$ is the sought affine prolongation of $G$ to $X$ with connected fibers and with big-cell given by the image of the big cell of $\mathfrak{G}_{_{\tt{der}}} \times Z_{_{G,X}}^0$.

\section{The recursive step over higher dimensional bases} 
For proving Theorem \ref{mt}, \S \ref{par-are-affine} is the base case. The notations are as in \S \ref{intro}. To prove the general case, in what follows, we apply the recursive step of J.~K. Yu to the restrictions on the completed local rings at the primes of height $1$ given by the $D_{_j}$. By taking the completions of the local rings at the primes of height $1$, we get into the situation described above on each $X_{_i}$. In this setting, we may assume the existence of $\Gfr_{_{f_e}}$ from Bruhat-Tits theory and therefore that of $H_{_{\kappa}}$ as in \eqref{Hkappa} without apppealing to \cite{yu}. Thus, we get closed subgroups of the $H_{_\kappa}$-type in the closed fibre of the restriction of the group scheme to $X_{_i}$. The new feature is then to extend this process to higher dimensions. We will prove the following key inductive result in the next sections.

\begin{thm} \label{first} Let $G$ be a split connected reductive algebraic group. Let $\Gfr \rightarrow X$ be a connected smooth affine group scheme with a big-cell on a smooth quasi-projective scheme $X$ as in Theorem \ref{mt}. Let  $D$ be one of the divisors with generic point $\text{Spec}(\kappa)$ which we fix. Let 
$\Gfr_{_D} = \Gfr|_{_{D}}$ denote the restrictions of $\Gfr$ to $D$ and let 
$ \Gfr_{_{D,\kappa}}$ denote its restriction to the generic point of $D$. Let $H_{_{\kappa}} \subset \Gfr_{_{D,\kappa}}$ be a closed connected smooth subgroup scheme as in \eqref{Hkappa}. Then $H_{_{\kappa}}$ extends to $D$ as an affine closed connected smooth subgroup scheme $\mathfrak{H}_{_{D}}$ of $\Gfr_D$. Further, the big-cell of $\Gfr$ canonically induces one on the dilatation of $\Gfr$ along  $\mathfrak{H}_D$.
\end{thm}
Let us see how after \S \ref{par-are-affine}, Theorem \ref{mt} reduces to \ref{first}.



\subsubsection{Reduction to the case $D$ is a single divisor assuming Theorem \ref{first}} \label{redtoonedivisor} 
Let $\Gfr \rightarrow X$ be a smooth affine group scheme on a smooth quasi-projective scheme $(X,D)$, with  $D = \cup_{j=0}^{n} D_{_j}$ a s.n.c.divisor. 
 Along the component $D_j$, suppose we are given a connected closed and smooth $D_j$-subgroup scheme $H_{_{j,\kappa}} \subset \Gfr_{_{D_j,\kappa}}$. By Theorem \ref{first}, $H_{_{j,\kappa}}$ extends to $D_j$ as a connected, closed affine smooth subgroup scheme 
 \begin{equation} \mathfrak{H}_{_j} \subset \Gfr|_{D_{_j}}.
 \end{equation}
 By \cite[Theorem 3.2 (1), (3), (5), Lemma 3.1 (2)]{mayeux},  the  dilatation   
\begin{equation} \Gfr^{\mathfrak{H}_{_j}}
\end{equation} of $\Gfr$ along $\mathfrak{H}_{_j}$ is a connected, smooth, affine group scheme on $X$ with an affine morphism
\begin{equation}
\Gfr^{\mathfrak{H}_{_j}} \rightarrow \Gfr,
\end{equation}
which is an isomorphism outside of $X_{_j}$. Let \begin{equation} \Gfr^{{\mathfrak{H}}} \ra \Gfr 
\end{equation} be the smooth affine group scheme obtained by  dilatations along each $\mathfrak{H}_{_j}$ taken one by one in any order. Since $\Gfr^{{\mathfrak{H}}}$ is affine with connected fibers, by considering its global section we check that it also interpolates the generic dilatation groups on $X_{_i}$, i.e., with notations as above we have


\begin{equation}
\Gfr^{{ \mathfrak{H}}}|_{_{X_{_i}}} \simeq \Gfr^{{\mathfrak{H}_{_i}}} |_{_{X_{_i}}}.
\end{equation} 
This completes the construction. This group scheme is unique by Theorem \ref{uniquenessraynaud}.


\subsection{Relation with $\mathfrak{G}_{(f_0,\cdots,f_n)} \ra X$ of \cite{bp}} \label{crosscheckbigcell} Let $\widehat{U}_{_{\text{ht} \leq 1}}$ denote the pushout of $X_{_o}$ and the $X_{_i}$. By \cite[Section 6.2, Proposition 6.4 (d)]{blr}, the restriction $\Gfr^{{\mathfrak{H}}}|_{\widehat{U}_{_{\text{ht} \leq 1}}}$ 
extends to an open subset, say, 
\begin{equation} \label{bigopen} U_{_{\text{ht} \leq 1}} \subset X
\end{equation}  containing  $\widehat{U}_{_{\text{ht} \leq 1}}$. The complement of $U_{_{\text{ht} \leq 1}}$ in $X$ has codimension two. By definition, we have a canonical isomorphism between the restrictions of $\Gfr^{{\mathfrak{H}}}$ and $\mathfrak{G}_{(f_0,\cdots,f_n)}$ to ${U}_{_{\text{ht} \leq 1}}$. Further this isomorphism respects the big-cell structures on both group schemes.
Since both group schemes are smooth with connected fibers and extend to $X$, this isomorphism extends to all of $X$ by Theorem \ref{uniquenessraynaud}. Further the isomorphism between the big-cells extends as well.

\section{Proof of \eqref{first} when $\text{dim}(X) = 2$} \label{dim=2}
Subscripts will denote restrictions or to remind us of the place of definition. Thus the $H(\kappa)$ of \cite{yu} is denoted $H_{_\kappa}$.

\begin{proof}
We now place ourselves in the setting of J.-K. Yu \cite[Lemma 8.1.2, page 250]{yu} where we assume that all the concave functions $f$ we will work with are such that $f(0) = 0$.  Let $D$ be a {\it single} smooth divisor with generic point 
\begin{equation} \label{genericpoint} \zeta=\text{Spec}(\kappa).
\end{equation} We denote the completed local ring at the generic point $\zeta$ by
\begin{equation} \label{tubularO}
\cO:= \hat{\cO}_{_{X,\zeta}}.
\end{equation} 
We will identify $\text{Spec}(\kappa)$ with the closed point of $\mathcal{O}$.
Let the restriction $\Gfr|_{\text{Spec}(\cO)}$ correspond to the {\it starting} concave function
\begin{equation}\label{psietc} \psi. \end{equation}
We have replaced the earlier $f_s$ by $\psi$ to avoid excessive subscripts.
Let $\Gfr_{_{\kappa}}$ denote the restriction of $\Gfr_{\psi,\cO}$ to $\kappa$. Let $\mathfrak{T}_{_{\kappa}}$ denote the restriction of the split maximal toral subgroup scheme $\mathfrak{T} \hookrightarrow \mathfrak{G}_{_{\psi,\cO}}$ to $\zeta$.

Let $\mathfrak{G}^{^{\tt u}}_{_{\kappa}} \triangleleft \mathfrak{G}_{_{\kappa}}$ be the unipotent radical. The first goal is to prolongate the normal subgroup $\mathfrak{G}^{^{\tt u}}_{_{\kappa}}$,  to $D$ as a {\em closed normal subgroup scheme with connected fibers} of $\Gfr_D$. 

For $a \in \boldsymbol{\Phi}$, let $U_{_{\kappa,a}}$ be the corresponding root group. 
The action of the roots of $\mathfrak{T}_{_{\kappa}}$ on $\mathfrak{G}^{^{\tt u}}_{_{\kappa}}$ give rise to a {\it close} subset (see \cite[1.1.4]{bt2})
\begin{equation}\label{phigtu}\boldsymbol{\Phi}_{_G}^{^{\tt u}} \subset \boldsymbol{\Phi},\end{equation}
which means that, for $a$ and $b$ in $\boldsymbol{\Phi}_{_G}^{^{\tt u}}$, the positive cone generated by $a$ and $b$ in $\boldsymbol{\Phi}$ is contained in $\boldsymbol{\Phi}_{_G}^{^{\tt u}}$. Thus, $\boldsymbol{\Phi}_{_G}^{^{\tt u}} \subset \boldsymbol{\Phi}$ correspond to the set of roots such that the root groups $U_{_{\kappa,a}}$ define the unipotent radical  $\mathfrak{G}^{^{\tt u}}_{_{\kappa}}$. We therefore have an exact sequence:
\begin{equation}\label{leviexactseq}
(1) \ra \mathfrak{G}^{^{\tt u}}_{_{\kappa}} \ra \Gfr_{_{\kappa}} \to \mathfrak{G}^{^{\tt red}}_{_{\kappa}} \to (1)
\end{equation}
where $\mathfrak{G}^{^{\tt red}}_{_{\kappa}}$ is the canonical reductive quotient.

Let us denote the neutral component of the schematic closure of $\mathfrak{G}^{^{\tt u}}_{_{\kappa}}$ (resp. $U_{_{\kappa,a}}$) in $\mathfrak{G}_{_{D}}:=\mathfrak{G}|_{_{D}}$ by 
\begin{equation} \label{HUnipotentExtnD}  \left(\mathfrak{G}^{^{\tt u}}_{_{D}} \right)^{^{\tt 0}} \quad \left(\text{resp.} \quad  \mathfrak{U}^{^{\tt 0}}_{_{D,a}} \right).  
\end{equation}
Since $D$ is smooth of $\text{dim}~1$, these extensions are flat. Hence, by \cite[1.2.7]{bt2}, these are closed subgroup schemes of $\mathfrak{G}_{_{D}}$ and hence {\em affine}.  

We need to get {\em smoothness}. The direct sum decomposition
\begin{equation} 
\text{Lie}~\Gfr^{^{\tt u}}_{_{\kappa}} = \bigoplus_{a \in \boldsymbol{\Phi}_{_G}^{^{\tt u}}} \text{Lie}~U_{_{\kappa,a}},
\end{equation} 
extends to a direct sum decomposition of locally-free $\cO_{_{D}}$-modules 
\begin{equation}
\text{Lie}~ \left( \mathfrak{G}^{^{\tt u}}_{_{D}}\right)^{^{\tt 0}} = \bigoplus_{a \in \boldsymbol{\Phi}_{_G}^{^{\tt u}}} \text{Lie}~\mathfrak{U}^{^{\tt 0}}_{_{D,a}}.
\end{equation} 

We now apply Definition \ref{Wfunctor} (see also \cite[\S 1.4.1]{bt2}) from the category of {\it nilpotent} locally-free $\cO_{_{D}}$-modules of finite rank to the category of unipotent $D$-group schemes (see \cite[\S 5.2]{bp} for properties) to the above to get an isomorphism of unipotent $D$-group schemes
\begin{equation*}
 \left(\mathfrak{G}^{^{\tt u}}_{_{D}} \right)^{^{\tt 0}}= W \left(\text{Lie}~\left(\mathfrak{G}^{^{\tt u}}_{_{D}} \right)^{^{\tt 0}} \right)  = W \left(\oplus_{a \in \boldsymbol{\Phi}_{_G}^{^{\tt u}}} \text{Lie}~\mathfrak{U}^{^{\tt 0}}_{_{D,a}}  \right)= \prod_{a \in \boldsymbol{\Phi}_{_G}^{^{\tt u}}} W \left(\text{Lie}~\mathfrak{U}^{^{\tt 0}}_{_{D,a}} \right) = \prod_{a \in \boldsymbol{\Phi}_{_G}^{^{\tt u}}}  \mathfrak{U}^{^{\tt 0}}_{_{D,a}}.
\end{equation*} 
Since for any $a \in \boldsymbol{\Phi}$, $\text{Lie}~\mathfrak{U}^{^{\tt 0}}_{_{D,a}}$ are rank one locally-free $\cO_{_{D}}$-modules, $ \mathfrak{U}^{^{\tt 0}}_{_{D,a}}$ are smooth group schemes. Whence, $ \left(\mathfrak{G}^{^{\tt u}}_{_{D}} \right)^{^{\tt 0}}$ is also {\em smooth}, and also {\ul{ closed}} \& hence {\em affine} as seen above. 

The Lie algebra bundle $\bigoplus_{a \in \boldsymbol{\Phi}_{_G}^{^{\tt u}}} \text{Lie}~U_{_{D,a}}$ is an ideal in $\text{Lie}~ \mathfrak{G}_{_D}$ over  $U_{_{\text{ht} \leq 1}} \cap D$ \eqref{bigopen} and hence on $D$. Thus $ \left(\mathfrak{G}^{^{\tt u}}_{_{D}} \right)^{^{\tt 0}}$ is a normal subgroup scheme of $\mathfrak{G}_{_D}$.

Since we are over $D$ which smooth of $\text{dim}~1$, by Theorem \ref{rem9.3}, the quotient $\Gfr_{_D}/ \left(\mathfrak{G}^{^{\tt u}}_{_{D}} \right)^{^{\tt 0}}$ is representable by a group scheme, which we denote by 
$$\Gfr_{_D}^{^{\tt r\'edg\'en}},$$
to indicate that it is only {\it generically reductive}. This is a connected smooth affine group scheme over $D$ which prolongates the reductive quotient \eqref{leviexactseq} $\Gfr_{_{\kappa}}^{^{red}}$ and we get an exact sequence:
\begin{equation}\label{leviexactseqD}
(1) \ra \left( \mathfrak{G}^{^{\tt u}}_{_{D}} \right)^{^{\tt 0}} \ra \Gfr_{_{D}} \to \Gfr_{_D}^{^{\tt r\'edg\'en}} \to (1).
\end{equation}

As in \eqref{Hkappa}, let $H_{_\kappa} \subset \Gfr_{_{\psi,\kappa}}$ be a connected Zariski closed subgroup  for dilatation of $\Gfr_{_{\psi, \cO}}$ and let 
\begin{equation}\label{epsilon} \epsilon
\end{equation}
 be the concave function one {\it ends} with after dilatation along $H_{_\kappa}$. 
 Thus, 
 \begin{equation}\label{Hasimage} H_{_\kappa}= \text{Image} \left( \Gfr_{_{\epsilon,\cO}} \ra \Gfr_{_{\psi,\cO}} \right)|_{\kappa}.
 \end{equation}
The goal is to show that the identity component of its schematic closure is a {\em smooth, affine and closed} subgroup scheme of $\Gfr_{_D}$. 

Since the morphism $\Gfr_{_{\epsilon,\cO}} \ra \Gfr_{_{\psi, \cO}}$ respects the big-cell structure, the roots of $H_{_\kappa}$ relative to $\mathfrak{T}_{_{\kappa}}$ may be identified with  \begin{equation} \label{rootHkappa}
\boldsymbol{\Phi}_{_{H_{_\kappa}}} :=  \big\{ a \in \boldsymbol{\Phi} | \epsilon(a)=\psi(a) \big\}.
\end{equation}  Let $H^{^{\tt u}}_{_{\kappa}}$ denote the {\em unipotent radical} of $H_{_{\kappa}}$ with root system
\begin{equation}
\boldsymbol{\Phi}_{_H}^{^{\tt u}} \subset \boldsymbol{\Phi}_{_{H_{_\kappa}}}.
\end{equation} 

For $a \in \boldsymbol{\Phi}_{_H}^{^{\tt u}}$, let $U_{_{\kappa,a}}$ denote the corresponding root subgroup.
Let us denote the neutral component of the schematic closure of $H^{^{\tt u}}_{_{\kappa}}$ (resp. $U_{_{\kappa,a}}$) in $\mathfrak{G}_{_{D}}:=\mathfrak{G}|_{_{D}}$ by 
\begin{equation}   \left(\mathfrak{H}^{^{\tt u}}_{_{D}} \right)^{^{\tt 0}} \quad \left(\text{resp.} \quad  \mathfrak{U}^{^{\tt 0}}_{_{D,a}} \right).  
\end{equation} 
Since $D$ is smooth of $\text{dim}~1$, these extension are flat. Hence, by \cite[1.2.7]{bt2}, these are closed subgroup schemes of $\mathfrak{G}_{_{D}}$ and hence {\em affine}.  

The {\em smoothness} of $\left(\mathfrak{H}^{^{\tt u}}_{_{D}} \right)^{^{\tt 0}}$ follows exactly from the arguments for $\left(\mathfrak{G}^{^{\tt u}}_{_{D}} \right)^{^{\tt 0}} $.

Let $H^{^{\tt red}}_{_{\kappa}}$ denote the {\em reductive or Levi quotient} of $H_{_{\kappa}}$; we may view this  as a connected, closed subgroup scheme of $\mathfrak{G}_{_{\psi,\kappa}}$ by taking a splitting of the Levi quotient. Let $\boldsymbol{\Phi}_{_H}^{^{\tt red}}$ denote the root system of $H^{^{\tt red}}_{_{\kappa}}$. We denote by
$$ \left(\mathfrak{H}_{_D}^{^{\tt r\'edg\'en}} \right)^{^{\tt 0}}$$ 
the neutral component of the schematic closure of $H^{^{\tt red}}_{_{\kappa}}$ in $\mathfrak{G}_{_{D}}$.  As before, it is flat because $D$ is smooth of $\text{dim}~1$ and hence a group scheme by \cite[1.2.7]{bt2}. 

Since $\mathfrak{T}_{_{\kappa}} \hookrightarrow H^{^{\tt red}}_{_{\kappa}}$ is the split maximal torus of $\mathfrak{G}_{_{\kappa}}$,  we may choose the split maximal torus   $$\mathfrak{T}_{_{D}} \hookrightarrow \mathfrak{G}_{_{D}}$$
to  extend $\mathfrak{T}_{_{\kappa}}$. The big cell of $ \left(\mathfrak{H}_{_D}^{^{\tt r\'edg\'en}} \right)^{^{\tt 0}}$ identifies with 
 $$ \mathfrak{T}_{_D} \times \prod_{a \in \boldsymbol{\Phi}_{_H}^{^{\tt red}}}  \mathfrak{H}^{^{\tt 0}}_{_{D,\tt u,a}}.$$
 
 
 Since $\mathfrak{T}_{_D}$ and $ \mathfrak{H}^{^{\tt 0}}_{_{D,\tt u,a}}$ are smooth, \ $\left(\mathfrak{H}_{_D}^{^{\tt r\'edg\'en}} \right)^{^{\tt 0}}$ is also \ul{smooth} by \cite[2.2.5 (ii)]{bt2} and \ul{affine} by \cite[2.2.5 (iii)]{bt2}. Let us  denote the neutral component  of the schematic closure of $H_{_{\kappa}}$ in $\mathfrak{G}_{_{D}}$ by
\begin{equation} \label{H_kextD}  \mathfrak{H}^{^{\tt 0}}_{_D}.
\end{equation}  
It is flat hence a group scheme and therefore a subgroup scheme of $\mathfrak{G}_{_{D}}$. Further,  scheme-theoretically, $ \mathfrak{H}^{^{\tt 0}}_{_D}$ is a product of $\left(\mathfrak{H}_{_D}^{^{\tt r\'edg\'en}} \right)^{^{\tt 0}}$ and $ \left(\mathfrak{H}^{^{\tt u}}_{_{D}} \right)^{^{\tt 0}}$. Whence, it is  affine, connected, closed and smooth.
 
Finally, we deduce the big-cell structure. Gluing the group schemes $G \times X_{_o}$ with those on $X_{_0}$ and $X_{_1}$ given by the  concave functions $f_{_0}$ and  $\epsilon$ via $t_i \in T \left(X_{_o} \times_X X_{_i} \right)$ for $i=0,1$, we have a BT group scheme on the open subset $U_{_{\text{ht} \leq 1}}$ \eqref{bigopen} together with the big-cell structure given by root subgroup schemes corresponding to $r \in \boldsymbol{\Phi}$ and the prolongation to $U_{_{\text{ht} \leq 1}}$ of the maximal torus $T$ of $G$. We denote this group scheme as
\begin{equation}
\mathfrak{G}_{_{{\text{ht} \leq 1}}}.
\end{equation}
This group scheme is isomorphic to the restriction to $U_{_{\text{ht} \leq 1}}$ of the dilatation of $\mathfrak{G}$ by $ \mathfrak{H}^{^{\tt 0}}_{_D}$. Taking schematic closure of the root subgroups of  $\mathfrak{G}_{_{_{\text{ht} \leq 1}}}$ we get the sought big-cell structure.
 
\end{proof}


\section{When $\text{dim}(X) \geq 3$}
When $\text{dim}(X) = 2$, counter-examples in \cite[3.2.14 and 3.2.15]{bt2} show that taking schematic closures does not automatically prolongate the group structure to the limit. However, in this setting since the divisor $D$ corresponds to a Dedekind ring, taking schematic closures will be {\em flat} once generically so. Thus, smooth prolongations of smooth subgroup schemes are obtained relatively easily.

As can be imagined, novel issues emerge when the dimension of $X$ is strictly bigger than $2$. Since the divisor $D$ is now of dimension at least $2$, schematic closures are harder to manoeuvre, primarily due to problems arising out of {\em flatness}. The case handled in \cite[\S 6, \S 7]{bp} proves only quasi-affineness. In our framework, the degeneration along the divisor $D$ is of a maximal rank subgroup scheme $H_{_{\kappa}}$ of a {\em non-reductive} group scheme $\mathfrak{G}_{_{D}}$ and we seek a prologation as a  closed connected smooth subgroup scheme to reach our goal of proving affineness. 

\subsubsection{Towards the completion  of Theorem \eqref{first}}
\begin{proof}
With notations as in Theorem \ref{mt}, let $\{D_{_i} \}_{i=0}^n$ denote the set of divisors. By \S \ref{redtoonedivisor}, we may assume that the complement of $X_o \subset X$ is a {\it single} smooth divisor $Y:=D_{_n} \subset X$  and we work with the generic point $\zeta$ of $Y$. We denote the completed local ring at the generic point $\zeta$ by
\begin{equation} \label{tubularO1}
\cO:= \hat{\cO}_{_{X,\zeta}}
\end{equation} 
and the residue field by $\kappa$.
We remark that $Y$ in turn is a smooth scheme with normal crossing divisors given, for $i=0,\cdots, n-1$, by 
\begin{equation} \label{divisorsD} 
Y \cap D_{_i}. 
\end{equation}
We denote the generic point of the component $Y \cap D_{_i}$ by
\begin{equation} \zeta'_{_i}.
\end{equation} Let us set 
\begin{eqnarray}\label{basicnots} B_{_i} & := & \mathcal O_{_{Y,\zeta'_{_{i}}}}, \\ 
 Y_{_i} & := & \spec \left(\hat{B}_{_i} \right), \\ 
Y_{_o} & := & Y - \cup_{_{i = 0}}^{^{n-1}} \{Y \cap D_{_i}\}, \\
Y^{^{(\i)}} & := & Y_{_o} \bigcup \cup_{_{i = 0}}^{^{n-1}} \{ {Y}_{_i} \}.
\end{eqnarray} 
The codimension of the complement of $Y^{^{(\i)}}$ is at least two in $Y$.
 
Thus, with $Y$, we are again in the setting of Theorem \ref{mt}. We may formulate the induction hypothesis that Theorem \ref{mt} holds for all smooth schemes of dimension less than that of $X$. Section \ref{dim=2} constitutes the base case of this induction on dimension. Since $\text{dim}(Y) < \text{dim}(X)$, we may assume that Theorem \ref{mt} holds for the above configuration.

Let the notations $\kappa$, $\cO$, $\psi$, $\epsilon$, $\mathfrak{G}$, $\mathfrak{T} \hookrightarrow \mathfrak{G}$, $\mathfrak{G}_{_{\kappa}}$, $\mathfrak{T}_{_{\kappa}}$,  $H_{_{\kappa}}$, $H^{^{\tt u}}_{_{\kappa}}$, $H^{^{\tt red}}_{_{\kappa}}$, $\boldsymbol{\Phi}_{_{H_{_{\kappa}}}}=\boldsymbol{\Phi}_{_H}^{^{\tt u}} \cup \boldsymbol{\Phi}_{_H}^{^{\tt red}}$ be as in the previous section with the only difference that they are adapted to the case $\dim(X) \geq 3$. As before, to complete the proof of \eqref{first}, our aim is to construct a {\em prolongation} of $H_{_{\kappa}}$ to an {\em affine, closed, connected and smooth} $Y$-subgroup scheme of $\Gfr_{_Y}$. We achieve this in \eqref{final equation}.


Our starting point is the group scheme constructed \S \ref{par-are-affine}.  Thus, $\mathfrak{G}$ has a \ul{big-cell structure} given by:
\begin{equation}
\mathfrak{T} \times \prod_{a \in \boldsymbol{\Phi}} \mathfrak{U}_{_{a}},
\end{equation}
where the $\mathfrak{U}_{_{a}}$ are \ul{connected}. 


\subsubsection{Step 1: {\bf{The prolongation of the unipotent radical and a Levi-like decomposition of $\mathfrak{G}_{_{Y}}$}}}  \label{section1} 
\noindent

The first aim is to prolongate the normal subgroup $\mathfrak{G}^{^{\tt u}}_{_{\kappa}} \triangleleft \mathfrak{G}_{_{\kappa}}$,  to $Y$ as a {\em closed connected normal subgroup scheme} of $\Gfr_Y$.

The action of the roots of $\mathfrak{T}_{_{\kappa}}$ on $\mathfrak{G}^{^{\tt u}}_{_{\kappa}}$ give rise to a {\it close} subset (see \cite[1.1.4]{bt2})
$$\boldsymbol{\Phi}_{_G}^{^{\tt u}} \subset \boldsymbol{\Phi}.$$
Set
\begin{equation} \label{gdu1}
\mathfrak{G}^{^{\tt u}}_{_{Y}} := \prod_{a \in \boldsymbol{\Phi}_{_G}^{^{\tt u}}} \mathfrak{U}_{_{a,Y}}.
\end{equation}
This is a flat subgroup scheme of $\mathfrak{G}_{_Y}$. Although we do not know yet if it is  {\em closed} and {\em normal} in $\mathfrak{G}_{_Y}$, it has the property that its restriction to any $Y_{_i}$ is a closed normal subgroup scheme.   Let us denote the quotient group scheme by
\begin{equation}\label{theunipquot}
 \mathfrak{G}^{^{\tt /u}}_{_{Y_{_i}}} := \mathfrak{G}_{_{Y_{_i}}} / \mathfrak{G}^{^{\tt u}}_{_{Y_{_i}}}.
\end{equation}
We now go on to show that the connected group scheme $\mathfrak{G}^{^{\tt /u}}_{_{Y_{_i}}}$ helps us prolongate the Levi decomposition of $\mathfrak{G}_{_\kappa}$ to $\mathfrak{G}_{_{Y_{_i}}}$.

From $\mathfrak{G}_{_{Y_{_i}}}$, the group scheme $\mathfrak{G}^{^{\tt /u}}_{_{Y_{_i}}}$ acquires the following cell-structure 
\begin{equation} \label{cellstructure}
\mathfrak{T} \mid_{_{Y_{_i}}} \times \prod_{a \in \boldsymbol{\Phi} - \boldsymbol{\Phi}_{_G}^{^{\tt u}}} \mathfrak{U}_{_{a}} \mid_{_{Y_{_i}}},
\end{equation}
which is compatible with the inclusion $\mathfrak{G}^{^{\tt u}}_{_{Y_{_i}}} \hookrightarrow \mathfrak{G}_{_{Y_{_i}}}$.

Consider the natural quotient morphism $$q_{_{\kappa}}: \mathfrak{G}_{_{\kappa}} \ra \mathfrak{G}^{^{\tt red}}_{_{\kappa}}$$ and let us fix a Levi-splitting
$$\lambda_{_{\kappa}}: \mathfrak{G}^{^{\tt red}}_{_{\kappa}} \ra \mathfrak{G}_{_{\kappa}}.$$
Let us denote the connected component of the schematic closure of $\lambda_{_{\kappa}} \left(\mathfrak{G}^{^{\tt red}}_{_{\kappa}} \right)$ in $\mathfrak{G}_{_{Y_{_i}}}$ as 
\begin{equation}\label{redgen}
\mathfrak{G}^{^{\tt r\'edg\'en}}_{_{Y_{_i}}}.
\end{equation}
The superscript `{\tt r\'edg\'en}' signifies {\it generically reductive}. By construction, the group scheme $\mathfrak{G}^{^{\tt r\'edg\'en}}_{_{Y_{_i}}}$ is \ul{flat} and  inherits a cell-structure from $\mathfrak{G}_{_{Y_{_i}}}$ which is isomorphic to \eqref{cellstructure}.

Consider the product homomorphism of group schemes on $Y_{_i}$
\begin{equation}
\pi: \mathfrak{G}^{^{\tt r\'edg\'en}}_{_{Y_{_i}}} \times 
\mathfrak{G}^{^{\tt u}}_{_{Y_{_i}}} \ra \mathfrak{G}_{_{Y_{_i}}}.
\end{equation}
When restricted to the cells of each of the group schemes, $\pi$ becomes an isomorphism. 
It follows by Theorem \ref{schcopen} that the above morphism is an isomorphism onto an open subgroup scheme of $\mathfrak{G}_{_{Y_{_i}}}$. Since $\mathfrak{G}_{_{Y_{_i}}}$ is connected, it follows that we have an isomorphism. This shows also that $\mathfrak{G}^{^{\tt r\'edg\'en}}_{_{Y_{_i}}}$ via $\mathfrak{G}_{_{Y_{_i}}}$ maps isomorphically onto $\mathfrak{G}^{^{\tt /u}}_{_{Y_{_i}}}$. In other words, the Levi decomposition $$\Gfr_{_{\kappa}} = \Gfr^{^{\tt u}}_{_{\kappa}} \ltimes \Gfr^{^{\tt red}}_{_{\kappa}}$$ over $\kappa$ prolongates to $Y_{_{i}}$ for every $i$, and hence to the whole of $Y^{^{(\i)}} \subset Y$. We denote this prolongated Levi decomposition as follows 
\begin{equation}\label{semidoverdvr} \Gfr_{_{Y}}|_{_{Y^{^{(\i)}}}} = \Gfr^{^{\tt u}}_{_{Y}}|_{_{Y^{^{(\i)}}}} \ltimes \Gfr_{_{Y}}^{^{\tt r\'edg\'en}}|_{_{Y^{^{(\i)}}}}. \end{equation}


We denote the quotient by
\begin{equation}
q_{_{Y^{^{(\i)}}}}: \mathfrak{G}_{_Y}|_{_{Y^{^{(\i)}}}} \rightarrow \Gfr_{_Y}^{^{\tt r\'edg\'en}} |_{_{Y^{^{(\i)}}}}.\end{equation}

Observe the restrictions $\mathfrak{G}_{_{Y_{_o}}}$ and $\mathfrak{G}_{_{Y_{_i}}}$ to $Y_{_o} \times_Y Y_{_i}$ are glued by transition functions that preserve the cell-structures. Therefore they lie in $T \left(Y_{_o} \times_Y Y_{_i} \right)$. We denote them by $s_{_i}$.  By the induction hypothesis, on $Y$, we have the {\em affine, connected and smooth} group scheme 
\begin{equation} \Gfr_{_Y}^{^{\tt r\'edg\'en}}
\end{equation}
interpolating the generically reductive quotients $\mathfrak{G}^{^{\tt r\'edg\'en}}_{_{Y_{_i}}}$ and $\mathfrak{G}^{^{\tt red}}_{_{\kappa}}$ of 
$\mathfrak{G}_{_{Y_{_i}}}$ and  $\mathfrak{G}_{_{\kappa}}$ respectively using $s_{_i}$ as gluing function. 


 Since $\mathfrak{G}_{_{Y}}$ is smooth and the codimension of the complement of $Y^{^{(\i)}}$ is at least two, by  Theorem~\ref{Corollaire IX. 1.4}, the morphism $q_{_{Y^{^{(\i)}}}}$ extends uniquely to $Y$. We denote the extension as
$$q_{_{Y}}: \mathfrak{G}_{_Y} \rightarrow \Gfr_{_Y}^{^{\tt r\'edg\'en}}.$$
Now, the group schemes $\text{ker} \left(q_{_{Y}} \right)$ and $\mathfrak{G}^{^{\tt u}}_{_{Y}}$ \eqref{gdu1} both prolongate $\text{ker} \left(q_{_{Y^{^{(\i)}}}} \right)$ which is smooth being isomorphic to $\prod_{a \in \boldsymbol{\Phi}_{_G}^{^{\tt u}}} \mathfrak{U}_{_{a,Y^{^{(\i)}}}}$. Thus, by Theorem \ref{uniquenessraynaud}, we deduce that
\begin{equation}
\mathfrak{G}^{^{\tt u}}_{_{Y}}=\text{ker}\left(q_{_{Y}} \right).
\end{equation}
This achieves the two properties of {\it closedness and normality} of the subgroup scheme $\mathfrak{G}^{^{\tt u}}_{_{Y}}$ of $\mathfrak{G}_{_{Y}}$.


Taking schematic closures over height $1$ primes, we may extend $\lambda_{_{\kappa}}$ to  
$$\lambda_{_{Y^{^{(\i)}}}}: \Gfr_{_Y}^{^{\tt r\'edg\'en}}|_{_{Y^{^{(\i)}}}} \rightarrow \mathfrak{G}_{_{Y}}|_{_{Y^{^{(\i)}}}}$$ over $Y^{^{(\i)}}$.

Since $ \Gfr_{_Y}^{^{\tt r\'edg\'en}}$ is smooth and the codimension of the complement of $Y^{^{(\i)}}$ is at least two, again by Theorem \ref{Corollaire IX. 1.4}, the morphism $\lambda_{_{Y^{^{(\i)}}}}$ extends to 
$$\lambda_{_{Y}}: \Gfr_{_Y}^{^{\tt r\'edg\'en}} \rightarrow \mathfrak{G}_{_{Y}}.$$

Using $\lambda_{_{Y}}$ and the normality of $\mathfrak{G}^{^{\tt u}}_{_{Y}}$, we construct the semi-direct product
\begin{equation}\label{semidoverdvr1}
\Gfr^{^{\tt u}}_{_{Y}} \ltimes \Gfr_{_Y}^{^{\tt r\'edg\'en}} \hra \Gfr_{_Y}.
\end{equation}
This is a smooth subgroup scheme prolongating  $\Gfr_{_{Y^{^{(\i)}}}}$ by \eqref{semidoverdvr}. Thus, $\Gfr_{_{Y^{^{(\i)}}}}$ gets two \ul{connected} smooth prolongations to $Y$. Therefore by Theorem \ref{uniquenessraynaud}, it follows that \eqref{semidoverdvr1} is an isomorphism of group schemes.

\subsubsection{ Step 2: {\bf Prolongation of $H^{^{\tt red}}_{_{\kappa}} \hookrightarrow \mathfrak{G}^{^{\tt red}}_{_{\kappa}}$ to $Y$ as a closed connected and smooth subgroup scheme.}} \label{step2}
Set $\mathfrak{G}_{_{\psi}}:= \mathfrak{G}_{_\mathcal{O}}$ (see  \eqref{tubularO1}). 
Let $\mathfrak{G}_{_{\epsilon}}$ denote the dilatation of 
$\mathfrak{G}_{_{\psi}}$ along $H_{_{\kappa}}$. 
We get the morphism
$$\mathfrak{G}_{_{\epsilon}} \ra \mathfrak{G}_{_{\psi}}$$
of group schemes over $\mathcal{O}$. Over $\Spec(\kappa)$, this gives
$$\mathfrak{G}_{_{\epsilon, \kappa}} \ra H_{_{\kappa}} \hookrightarrow \mathfrak{G}_{_{\psi, \kappa}}.$$

By Theorem \eqref{mcninch} a) and b), we have
\begin{eqnarray}
\mathfrak{G}^{^{\tt red}}_{_{\epsilon,\kappa}}= H^{^{\tt red}}_{_{\kappa}}, \\
\mathfrak{G}^{^{\tt red}}_{_{\epsilon,\kappa}} \hookrightarrow \mathfrak{G}^{^{\tt red}}_{_{\kappa}}.
\end{eqnarray}
This is an inclusion of a {\em maximal rank subgroup} in the reductive group $\mathfrak{G}^{^{\tt red}}_{_{\kappa}}$. Hence, by Theorem \ref{mcninch} c), we get
$$H^{^{\tt red}}_{_{\kappa}} =  C_{_{\mathfrak{G}^{^{\tt red}}_{_{\kappa}}}} \left(Z \left(H^{^{\tt red}}_{_{\kappa}} \right)\right)^{^o}.$$

Let us have the notation:
\begin{equation}
Z_{_{\kappa}}:= Z\left( H^{^{\tt red}}_{_{\kappa}} \right).
\end{equation} 
Thus, $Z_{_{\kappa}}$ is a subgroup scheme of $\mathfrak{T}_{_{\kappa}}$. Since $\mathfrak{T}_{_{\kappa}}$ extends to $Y$ as a split maximal torus, $Z_{_{\kappa}}$ extends as well. We denote the extension as 
\begin{equation}
Z_{_{Y}} \hookrightarrow \mathfrak{T}_{_{Y}}.
\end{equation}
Taking centralizers, we now set 
\begin{equation} \label{HredgenY}
\mathfrak{H}_{_Y}^{^{\tt r\'edg\'en}}
 :=C_{_{\Gfr_{_Y}^{^{\tt r\'edg\'en}}}} \left(Z_{_{Y}} \right)^{^o}.
\end{equation} 
Since $\Gfr_{_Y}^{^{\tt r\'edg\'en}}$ is a smooth affine $Y$-group scheme and $Z_{_{Y}}$ is a subgroup scheme of multiplicative type, by Theorem \ref{sixthree}
\begin{equation} 
\mathfrak{H}_{_Y}^{^{\tt r\'edg\'en}} \hookrightarrow \Gfr_{_Y}^{^{\tt r\'edg\'en}}
\end{equation} 
is a {\em closed and smooth subgroup scheme}.

\subsubsection{ Step 3: {\bf An embedding of $H^{^{\tt u}}_{_{\kappa}}$ into a closed subgroup scheme  $N_{_{Y}}$ of $\mathfrak{G}_{_{Y}}$}} \label{step3}

Set 
\begin{equation}
\boldsymbol{\Phi}_{_H}^{^{\tt u}}:= \text{roots of $\mathfrak{T}_{_{\kappa}}$-action on $H^{^{\tt u}}_{_{\kappa}}$}.
\end{equation}
Thus, $$\boldsymbol{\Phi}_{_H}^{^{\tt u}} \subset \boldsymbol{\Phi}$$
 is close. Set
 \begin{equation}
 \mathfrak{H}^{^{\tt u}}_{_{Y}}:= \prod_{a \in \boldsymbol{\Phi}_{_H}^{^{\tt u}}} \mathfrak{U}_{_{a,Y}}.
 \end{equation}
Note that  $$\mathfrak{H}^{^{\tt u}}_{_{\kappa}} =  H^{^{\tt u}}_{_{\kappa}}.$$

\`A priori, this may not be a closed subgroup scheme of $\mathfrak{G}_{_{Y}}$. However,
$$\text{Lie}~\mathfrak{H}^{^{\tt u}}_{_{Y}} \subset \text{Lie}~ \mathfrak{G}_{_{Y}},$$
is a {\em direct summand}. For the adjoint action, define the functor
\begin{equation}
N_{_{Y}}:= \text{Stab}_{_{\mathfrak{G}_{_{Y}}}}  \left( \text{Lie}~\mathfrak{H}^{^{\tt u}}_{_{Y}} \subset \text{Lie}~ \mathfrak{G}_{_{Y}} \right).
\end{equation} 
Observe that 
\begin{equation} \label{isanideal}
\text{Lie}~N_{_{Y}} = \text{Stab}_{_{\text{Lie} \mathfrak{G}_{_{Y}}}}  \left( \text{Lie}~\mathfrak{H}^{^{\tt u}}_{_{Y}}  \right).
\end{equation} 
Thus, $\text{Lie}~\mathfrak{H}^{^{\tt u}}_{_{Y}}$ is an ideal in $\text{Lie}~N_{_{Y}}$. 
In particular, $H^{^{\tt u}}_{_{\kappa}}$ is normal in $N_{_{k}}$. Whence we have the inclusion
\begin{equation} \label{unipotentinclusion}
H^{^{\tt u}}_{_{\kappa}} \triangleleft N^{^{\tt u}}_{_{\kappa}}.
\end{equation}

Now $\text{Lie}~\mathfrak{H}^{^{\tt u}}_{_{Y}}$ is {\em essentially-free}, being a direct-sum of line bundles on $Y$. Whence, by \eqref{exemplee} it follows that 
$$N_{_{Y}} \hookrightarrow \mathfrak{G}_{_{Y}},$$
is a {\em closed subgroup scheme, which however need not be flat}. 

Let $N_{_{\kappa}}$ be the restriction of $N_{_{Y}}$ to the generic point $\spec \kappa$. We now {\em claim} that we have an  embedding of $H_{_{\kappa}}$ in $N_{_{\kappa}}$. Observe that by definition \begin{equation}
N_{_{\kappa}}= \text{Stab}_{_{\mathfrak{G}_{_{\kappa}}}} \left(\mathfrak{h}^{^{\tt u}}_{_{\kappa}} \right), \end{equation}
where  $$\mathfrak{h}^{^{\tt u}}_{_{\kappa}}:=\text{Lie}~H^{^{\tt u}}_{_{\kappa}} \quad \text{which equals} \quad  \text{Lie}~ \mathfrak{H}^{^{\tt u}}_{_{\kappa}}.$$

To prove the claim, we proceed as follows. We begin by recalling the `${\text{Ad}}$' morphism. Let $R$ be an arbitrary $\kappa$-algebra. 

Thus, $$\mathfrak{h}^{^{\tt u}}_{_{\kappa}}(R)= \text{Ker} \left( H^{^{\tt u}}_{_{\kappa}} \left(R[\epsilon] \right) \ra H^{^{\tt u}}_{_{\kappa}}(R) \right).$$
Let $g \in H_{_{\kappa}}(R)$ and $x \in \mathfrak{h}^{^{\tt u}}_{_{\kappa}}(R)$. Embedding $R \ra R[\epsilon]$ as constants, we define
$$i: H_{_{\kappa}}(R) \ra H_{_{\kappa}}(R[\epsilon]).$$
Then conjugating as elements in $H_{_{\kappa}}(R[\epsilon])$ we have $$\text{Ad}(g)(x):= i(g)\cdot x\cdot i(g)^{-1}.$$
Since $H^{^{\tt u}}_{_{\kappa}}$ is normal in $H_{_{\kappa}}$, this element belong to $H^{^{\tt u}}_{_{\kappa}}(R [\epsilon] )$. Since $x$ belongs to $\mathfrak{h}_{_{\kappa}}(R)$, this element also belongs to $\mathfrak{h}_{_{\kappa}}(R)$. Thus,
$$\text{Ad}(g)(x) \in \mathfrak{h}_{_{\kappa}}(R) \cap H^{^{\tt u}}_{_{\kappa}}(R[\epsilon]) \quad \text{which equals} \quad \mathfrak{h}^{^{\tt u}}_{_{\kappa}}(R),$$
showing that $H_{_{\kappa}} $ stabilizes $\mathfrak{h}_{_{\kappa}}^{^{\tt u}}$.
This shows that 
$$H_{_{\kappa}} \hookrightarrow N_{_{\kappa}}$$
is a closed subgroup scheme. 
Hence by the Levi-decompositions for either groups and \eqref{unipotentinclusion}, we deduce 
\begin{equation} \label{redinclusion}
H_{_{\kappa}}^{^{red}} \hookrightarrow N_{_{\kappa}}^{^{red}}.
\end{equation}


\subsubsection{Step 4: {\bf A canonical smooth subgroup scheme of $N_{_{Y}}$}} \label{SmoothReplacementStep}

Consider the quotient $N_{_{\kappa}} \twoheadrightarrow N^{^{\tt red}}_{_{\kappa}}$. We take a Levi-splitting 
\begin{equation}
\lambda_{_{N, \kappa}}: N^{^{\tt red}}_{_{\kappa}} \hookrightarrow N_{_{\kappa}}.
\end{equation} 
By taking schematic closure inside $N_{_{Y^{^{(\i)}}}} :=N_{_{Y}} |_{_{Y^{^{(\i)}}}}$ we get a subgroup scheme  
$$\lambda_{_{N, Y^{^{(\i)}}}}: N_{_{Y^{^{(\i)}}}}^{^{\tt r\'edg\'en}} \hra N_{_{Y^{^{(\i)}}}}.$$

Although $N^{^{\tt red}}_{_{\kappa}}$ is by construction a reductive group, the schematic closure $N_{_{Y^{^{(\i)}}}}^{^{\tt r\'edg\'en}}$ is only a {\em generically reductive group scheme}. By our basic assumption on the prolongation of the split maximal torus, we see that both $N_{_{Y^{^{(\i)}}}}^{^{\tt r\'edg\'en}}$ and $N_{_{Y^{^{(\i)}}}}$ are smooth on $Y^{^{(\i)}}$. Further, the inclusion morphism $\lambda_{_{N, Y^{^{(\i)}}}}$ preserves the cell structure on these group schemes. Hence, the transition functions lie in $T \left( Y_{_o} \times_Y Y_{_i} \right)$ for each $i$. By the induction hypothesis, via the inclusion $Y^{^{(\i)}} \subset Y$, we get
$$N_{_Y}^{^{\tt r\'edg\'en}}$$
as an affine connected smooth group scheme on $Y$ prolongating $N_{_{Y^{^{(\i)}}}}^{^{\tt r\'edg\'en}}$, which by Theorem \ref{uniquenessraynaud} is unique.

Under $\mathfrak{T}_{_{\kappa}}$-action on $N^{^{\tt u}}_{_{\kappa}} \hookrightarrow N_{_{\kappa}}$, let 
$$\boldsymbol{\Phi}_{_N}^{^{\tt u}} \subset \boldsymbol{\Phi}$$
be the corresponding set of roots. Then $\boldsymbol{\Phi}_{_N}^{^{\tt u}} \subset \boldsymbol{\Phi}$ is a close subset. Set
\begin{equation}
\mathfrak{U}_{_{\boldsymbol{\Phi}^{^{\tt u}}_{_{N}},Y}} := \prod_{a \in \boldsymbol{\Phi}^{^{\tt u}}_{_{N}}} \mathfrak{U}_{_{a,Y}}.
\end{equation}

Consider the morphism
\begin{equation} \label{in}  i_{_{N}}: \mathfrak{U}_{_{\boldsymbol{\Phi}^{^{\tt u}}_{_{N}},Y}} \times N_{_Y}^{^{\tt r\'edg\'en}} \ra N_{_{Y}}.
\end{equation} 
Its restriction on $Y^{^{(\i)}}$ is an isomorphism.  Hence $i_{_{N}}$ is an injective morphism.
We set
\begin{equation}\label{thesm}
N_{_{Y}}^{^{sm}} := \text{Img} \left(i_{_{N}} \right).
\end{equation}
Thus, $N_{_{Y}}^{^{sm}}$ is an affine connected and smooth scheme, which however is not yet a {\em group scheme}. 
We set
\begin{equation}\label{inbar}
\overline{i_{_{N}}}: = \text{schematic closure $\overline{i_{_{N}} \left( \mathfrak{U}_{_{\boldsymbol{\Phi}^{^{\tt u}}_{_{N}},Y}} \times N_{_Y}^{^{\tt r\'edg\'en}} \right)}$ of $N_{_{Y}}^{^{sm}}$ in $N_{_{Y}}$}.
\end{equation} 
Thus, $\overline{i_{_{N}}}$ is {\ul{affine}} because $N_{_Y}$ is so. By Theorem \ref{schematiccopen}, the inclusion
$$j: \Img(i_{_{N}}) \subset \overline{i_{_{N}}}$$
is an open immersion. It is also an affine morphism since both $N_{_{Y}}^{^{sm}}$ and  $\overline{i_{_{N}}}$ are affine. 

We now want to show that  $$\text{codim}_{_{\overline{i_{_{N}}}}} \left( \overline{i_{_{N}}} \setminus \Img \left(i_{_{N}} \right) \right) \geq 2.$$ Since the morphism $i_{_{N,Y^{^{(\i)}}}}$ is an isomorphism we have,
$$\text{codim}_{_{\overline{i_{_{N}}}}} \left( \overline{i_{_{N}}} \setminus \Img \left(i_{_{N}} \right) \right) = \text{codim}_{_{N_{_Y}}} \left( \overline{i_{_{N}}} \setminus \Img \left(i_{_{N}} \right) \right) \geq \text{codim}_{_{N_{_Y}}} \left( \overline{i_{_{N}}} \setminus \Img \left(i_{_{N ,Y^{^{(\i)}}}} \right) \right).$$

Since the fiber dimension of the closed $Y$-subscheme $\overline{i_{_{N}}}$ is bounded above by the fiber dimension  of $N_{_{Y}}$, we have

$$\text{codim}_{_{N_{_Y}}} \left( \overline{i_{_{N}}} \setminus \Img \left(i_{_{N ,Y^{^{(\i)}}}} \right) \right)  \geq \text{codim}_{_{N_{_Y}}} \left( N_{_{Y}} \setminus \Img \left(i_{_{N,Y^{^{(\i)}}}} \right) \right).$$

 Although $N_{_Y}$ may not be a flat subgroup scheme of $\mathfrak{G}_{_{Y}}$, all irreducible components of any fiber of $N_{_{Y}}$  have the same dimension by Theorem \ref{ex6b}. Thus we have

$$  \text{codim}_{_{N_{_Y}}} \left( N_{_{Y}} \setminus \Img \left(i_{_{N,Y^{^{(\i)}}}} \right) \right) = \text{codim}_{_{Y}} \left( Y \setminus Y^{^{(\i)}} \right) =2.$$
This proves our claim on the codimension. 

On the other hand, Theorem \ref{van} says that  $$\text{if} \quad \Img \left(i_{_{N}} \right) \neq \overline{i_{_{N}}}, \quad \text{then}  \quad \text{codim}_{_{\overline{i_{_N}}}} \left( \overline{i_{_N}} \setminus \Img \left(i_{_N} \right) \right) =1.$$ By this mismatch, we deduce the equality
$$\Img \left(i_{_{N}} \right) = \overline{i_{_{N}}}.$$
Consequently, 
\begin{equation} \label{Ninc}
N_{_{Y}}^{^{sm}} \hookrightarrow \mathfrak{G}_{_{Y}}
\end{equation} 
is a connected, smooth, {\ul {closed subscheme}}. Since $N_{_{Y}}^{^{sm}}$ is closed \& flat on $Y$ and since its restriction to $Y^{^{(\i)}}$ is a subgroup scheme of $\mathfrak{G}_{_Y}|_{_{Y^{^{(\i)}}}}$,  it acquires a \ul{subgroup scheme structure} from $N_{_{Y}}$ by Theorem \ref{schematiccgroup}.

 Observe that
\begin{equation} \label{Liesm=Lieny}
\text{Lie}~N_{_{Y}}^{^{sm}} = \text{Lie} ~N_{_{Y}}.
\end{equation}

Having achieved this, we will henceforth have the notation:
\begin{equation} \label{notachange}
\mathfrak{N}_{_Y} := N_{_{Y}}^{^{sm}}.
\end{equation}

Applying \S  \ref{section1} to  the affine connected smooth group scheme $\mathfrak{N}_{_{Y}}$ on $Y$, we get the isomorphisms
\begin{eqnarray}
\mathfrak{N}^{^{\tt u}}_{_{Y}}  & = & \mathfrak{U}_{_{\boldsymbol{\Phi}^{^{\tt u}}_{_{N}},Y}}, \\ \label{sdprodnd} 
\mathfrak{N}_{_{Y}}  & = & \mathfrak{N}^{^{\tt u}}_{_{Y}} \ltimes \mathfrak{N}_{_Y}^{^{\tt r\'edg\'en}},
\end{eqnarray}
where as before the superscript {\tt r\'edg\'en} denotes the fact  that the group scheme is {\em generically reductive}. Since $\text{Lie}~ \mathfrak{N}^{^{\tt u}}_{_{Y}}$ is the unipotent radical of $\text{Lie}~ \mathfrak{N}_{_{Y}}$, from \eqref{Liesm=Lieny} and \eqref{isanideal} it follows that we have the inclusion of Lie algebra bundles
\begin{equation} \label{Lieinclusion}
\text{Lie}~ \mathfrak{H}^{^{\tt u}}_{_{Y}} \subset \text{Lie}~ \mathfrak{N}^{^{\tt u}}_{_{Y}}.
\end{equation}

In place of $\mathfrak{G}_{_{Y}}$, applying \S \ref{step2} to the affine connected smooth  $Y$-scheme $\mathfrak{N}_{_{Y}}$, and by \eqref{redinclusion} and \eqref{HredgenY} we have the closed inclusion 
\begin{equation} \label{hdredndred} 
\mathfrak{H}_{_Y}^{^{\tt r\'edg\'en}} \hookrightarrow \mathfrak{N}_{_Y}^{^{\tt r\'edg\'en}}
\end{equation} 
of affine connected smooth group schemes.

\subsubsection{Step 5: {\bf Prolongation of $H^{^{\tt u}}_{_{\kappa}}$ as a closed subgroup scheme of $N^{^{\tt u}}_{_{Y}} $}}

Since $i_{_{N}}$ \eqref{in} restricts to an isomorphism on $Y^{^{(\i)}}$, the inclusion \eqref{Lieinclusion} is a summand. Applying Definition \ref{Wfunctor}, we obtain the closed immersion
\begin{equation} \label{hdundu} 
\mathfrak{H}^{^{\tt u}}_{_{Y}}=W \left( \text{Lie}~ \mathfrak{H}^{^{\tt u}}_{_{Y}}  \right) \hookrightarrow  W \left( \text{Lie}~\mathfrak{N}^{^{\tt u}}_{_{Y}}  \right)  = \mathfrak{N}^{^{\tt u}}_{_{Y}}, \end{equation} 
which realises $\mathfrak{H}^{^{\tt u}}_{_{Y}}$ as  a closed subgroup scheme of $\mathfrak{N}^{^{\tt u}}_{_{Y}}$. 

\subsubsection{Step 6: {\bf Prolongation of $H_{_{\kappa}}$ to $Y$ as a smooth closed subgroup scheme of $\mathfrak{G}_{_{Y}}$}}

We have
\begin{equation}
\mathfrak{H}^{^{\tt u}}_{_{Y}} \times \mathfrak{H}_{_Y}^{^{\tt r\'edg\'en}} \stackrel{\eqref{hdundu}, \eqref{hdredndred}}{\hookrightarrow} \mathfrak{N}^{^{\tt u}}_{_{Y}} \ltimes \mathfrak{N}_{_Y}^{^{\tt r\'edg\'en}} \stackrel{\eqref{sdprodnd}}{=} \mathfrak{N}_{_{Y}}  \stackrel{\eqref{Ninc}, \eqref{notachange}}{\hookrightarrow} \mathfrak{G}_{_{Y}}.
\end{equation} 
Restricting the action of $N_{_Y}^{^{\tt r\'edg\'en}}$ on $N^{^{\tt u}}_{_{Y}}$ to $\mathfrak{H}_{_Y}^{^{\tt r\'edg\'en}}$ leaves $\mathfrak{H}^{^{\tt u}}_{_{Y}}$ stabilized. Thus, the semi-direct product of $N_{_{Y}}$ restricts to the product $\mathfrak{H}^{^{\tt u}}_{_{Y}} \times \mathfrak{H}_{_Y}^{^{\tt r\'edg\'en}}$ allowing us to define
\begin{equation}\label{final equation}
\mathfrak{H}_{_{Y}}: = \mathfrak{H}^{^{\tt u}}_{_{Y}} \ltimes \mathfrak{H}_{_Y}^{^{\tt r\'edg\'en}}.
\end{equation} 
The above steps realise $\mathfrak{H}_{_{Y}}$ as a connected closed smooth subgroup scheme of $\mathfrak{G}_{_{Y}}$, which is also affine. This provides the desired prolongation of $H_{_\kappa}$.

The existence of the big-cell structure on the dilatation of $\Gfr$ by $\mathfrak{H}_{_Y}$ is as in dimension two. This completes the proof of \eqref{first}.
 
 \end{proof}

 \subsubsection{On the assumption of perfectness of the residue field}
 We make some remarks regarding this hypothesis in the articles \cite{bt2}, \cite{yu}, \cite{base} and \cite{bp}. The main results of \cite{bt2}  do not require the residue field to be perfect. In the framework of \cite{bp} and in this work, the residue fields of the local rings of divisors such as $A_{_i}=\mathcal O_{_{X, \zeta_{_{i}}}}$ will clearly not be perfect.
 
 
 However, the perfectness of residue field is a key hypothesis of our reference \cite{yu}.  Although for the most part in \cite{bp} the residue field need not be perfect, it was added at certain places for convenience, without verifying its essentiality. In retrospect the perfectness of the residue field turns out to be {\em  inessential} owing to the splitness of $G$ and to the geometrical methods of \cite{base}, \cite{bp} and the present work.

\appendix
\section{Some key results used in the article}
Throughout this work, we refer to this section for all the classical results we use.
 
\begin{thm} \label{Corollaire IX. 1.4} \cite[Corollaire IX.1.4, page 130]{raynaud} Let $S$ be a scheme, $G$ and $H$ be two $S$ group schemes, with $G$ smooth on $S$ with connected fibers, and $H$ locally of finite type. We assume that $S$ is locally noetherian (resp. that $S$ is a Krull scheme and that $G$ is locally isomorphic to an open subset of $S[T_1,\cdots,T_n]$). Let $U$ be an open subset of $S$ containing the points of depth $\leq 1$ (resp. the points of codimension $\leq 1$) and $f^{U}: G_{_U} \ra H_{_U}$ a $U$- homomorphism. Then $f^{U}$ prolongs, uniquely, to a $S$-homomorphism $f: G \ra H$.
\end{thm} 

\begin{thm} \label{uniquenessraynaud} \cite[Corollaire IX.1.5]{raynaud} Let $S$ be a noetherian scheme, $U$ a open subset of $S$ containing points of depth $\leq 1$, $G_{_U}$ a $U$ smooth group scheme, with connected fibers. Then, upto a $S$-isomorphism, there exists at most one $S$ group scheme $G$, smooth on $S$, with connected fibers, which prolongs $G_{_U}$.
\end{thm}


\begin{defi} \label{Wfunctor} \cite[D\'efinition 4.6.1, \S 4.6 Module functors in the category of schemes]{sga3} Let $S$ be a scheme. For every $\cO_S$-module $\cF$, we denote ${\bf{V}}(\cF)$ and ${\bf{W}}(\cF)$ the contravariant functors on $(\text{Sch})/_S$ defined by:
\begin{eqnarray*}
{\bf{V}}(\cF)(S') & = & \text{Hom}_{\cO(S')} \left(\cF \otimes_{\cO_S} \cO_{S'},\cO_{S'} \right), \\
{\bf{W}}(\cF)(S') & = & \Gamma \left(S', \cF \otimes_{\cO_S} \cO_{S'} \right).
\end{eqnarray*} 
\end{defi}  

\begin{thm} \label{rem9.3} \cite[Expose $VI_{B}$, Remarque 9.3 b) $3^{\circ}$]{sga3} It is not excluded that $G/G'$ be representable on the other hand, when $G$ and $G'$ are of finite presentation on $S$, and that $G'$ is {\ul{flat}} on $S$ and {\ul{closed}} in $G$. Under these hypotheses, one knows that $G/G'$ is representable in the following particular cases:
\begin{enumerate}
\item[$1^{\circ}-$] $S$ is the spectrum of an artinian ring,
\item[$2^{\circ}-$] $G'$ is proper on $S$ and $G$ is quasi-projective on $S$,
\item[$3^{\circ}-$] $S$ is locally noetherian of dimension $1$.
\end{enumerate}
 \end{thm} 

\begin{thm} \label{sixthree} \cite[Expos\'e XIX, Proposition 6.3]{sga3} Let $S$ be a scheme, $G$ a smooth $S$ scheme of finite presentation on $S$, $Q$ a sub-torus of $G$. Then ${\ul{\text{Centr}}}_{_G}(Q)$ and ${\ul{\text{Norm}}}_{_G}(Q)$ are representable by closed and smooth subgroup schemes (and thus of finite presentation) on $S$.
\end{thm} 
By \cite[Expos\'e XI 6.11]{sga3}, ${\ul{\text{Centr}}}_{_G}(Q)$ and ${\ul{\text{Norm}}}_{_G}(Q)$ are representable by closed subgroup schemes of finite presentation of $G$. These are smooth by \cite[Expos\'e XI 2.4 and 2.4 bis]{sga3}. We refer the reader also to  \cite[Expos\'e XI, Cor. 5.3 or Thm 6.2]{sga3} for a similar result.

\begin{defi} \cite[Expos\'e VI$_B$ D\'efinition 6.2.1]{sga3} Let $f:X \ra S$ be a morphism of schemes. We say that $f$ is {\ul{essentially free}} on $S$, if one can find a covering of  $S$ by affine open $S_i$, for every $i$ a $S_i$-scheme $S_i'$ which is affine and faithfully flat on $S_i$, and a covering $(X'_{ij})_{j}$ of $X_i'=X \times_S S_i'$ by the open affine $X'_{ij}$, such that for every $(i,j)$, the ring of $X_{ij}'$ is a projective module over the ring of $S_i'$.  
\end{defi}

\begin{ex} \label{exemplee} \cite[Expos\'e VI$_B$ Exemples 6.2.4 e)]{sga3} When $G$ is a $S$-group scheme acting on a $S$-scheme $X$
$$q: G \ra \ul{\text{Aut}}_S(X),$$
the kernel of $q$ (`` the subgroup of $G$ acting trivially") is a {\ul{closed}} subscheme of $G$ provided that $X$ be essentially free and separated over $S$ (example c)),....

Let $Y$ and $Z$ be subschemes of $X$; let us consider the subfunctor ${\ul{\text{Transp}}}_G(Y,Z)$ of $G$ (``transporter of $Y$ in $Z$'') whose points with values in $T$ over $S$ consist of $g \in G(T)$ such that the corresponding automorphism of $X_T$ satisfies $g(Y_T) \subset Z_T$, i.e., induces a morphism $Y_{_T} \ra X_{_T}$ factoring through $Y_{_T} \ra Z_{_T}$. Thus,: {\em if $Y$ is essentially free on $S$, and $Z$ is closed in $X$, then $\ul{\text{Transp}}_G(Y,Z)$ is a closed subscheme of $G$} (example a).

One can also consider {\em the strict transporter of $Y$ in $Z$}, whose points with values in $T$ on $S$ are the $g \in G(T)$ such that $g(Y_T) =Z_T$,.... It follows, {\em if $Y$ and $Z$ are essentially free on $S$ and closed in $X$, the strict transporter of $Y$ in $Z$ is a {\ul{closed}} subscheme of $G$.} 
\end{ex}

\begin{prop}\label{ex6b} \cite[Proposition 4.2 a) Expos\'e VI$_B$]{sga3} Let $\pi: G \ra S$ be a $S$-scheme locally of finite presentation, equipped with a section $\epsilon$ and satisfying the following two conditions (which hold if $G$ is a $S$-group, by 1.5 and IV$_A$ 2.4.1 (loc. cit.) ):
\begin{enumerate}
\item[a)] For every $s \in S$ and every $x \in G_s$, we have
$$\text{dim}~G_s= \text{dim}_x~G_s.$$
(or equivalently (1.5), for every $s \in S$, all the irreducible components of $G_s$ have the same dimension).
\item[b)] For every $s \in S$, denoting $G^\circ_s$ the connected component of $G_s$ containing $\epsilon(s)$, $G^\circ_s$ is geometrically irreducible.
\end{enumerate}
\end{prop} 


\begin{thm} \label{van} \cite[Corollaire 21.12.7]{ega4} Let $X$ be a noetherian prescheme, $U$ an induced sub-prescheme on an open of $X$, $j: U \ra X$ the canonical immersion; suppose that $j$ is an affine morphism. Then every irreducible component of $T=X-U$ is of dimension $\leq 1$ (and thus of codimension $1$ if $U$ is everywhere dense).
\end{thm}




\subsubsection{From \cite{bt2}}
In \cite[\S 1]{bt2}, $K$ is an infinite field and $A$ is a sub-ring of $K$ having $K$ as its field of fractions.

\begin{thm} \label{schematiccopen} \cite[\S 1.2.6]{bt2}: Let $f: \mathfrak{Y} \ra \mathfrak{X}$ be a $A$-morphism. The {\it schematic image} of $f$ (or by abuse of language of $\mathfrak{Y}$) in $\mathfrak{X}$ is the smallest closed subscheme $\mathfrak{Y}'$ of $\mathfrak{X}$ such that $f$ maps $\mathfrak{Y}$ in $\mathfrak{Y}'$. Its ideal is the kernel of $f^*$ and its support is the closure of the image $f(\mathfrak{Y})$ (cf. \cite[pp. 176-177]{ega1}). If $f$ is the canonical immersion of a subscheme (affine) $\mathfrak{Y}$ of $\mathfrak{X}$, one also says that $\mathfrak{Y}'$ is the {\it schematic closure} of $\mathfrak{Y}$ in $\mathfrak{X}$. The scheme $\mathfrak{Y}$ is then an {\it open} subscheme of its schematic closure $\mathfrak{Y}'$ (loc. cit. 9.5.10). 
\end{thm}

\begin{thm} \label{schematiccgroup} \cite[\S 1.2.7]{bt2}: In particular, {\it if $Y$ is a closed $K$-subgroup scheme of $\mathfrak{X}_K$ and if the schematic closure $\mathfrak{Y}$ of $Y$ is a flat $A$-scheme, then $\mathfrak{Y}$ is a closed $A$-subgroup scheme of $\mathfrak{X}$.}
\end{thm}
 
 \begin{thm} \label{schcopen} \cite[\S 1.2.14]{bt2} Let $\mathfrak{G}$ and $\mathfrak{H}$ be two smooth $A$ group schemes and let $j: \mathfrak{G} \rightarrow \mathfrak{H}$ be a morphism. Suppose that $j_{_K}: \mathfrak{G}_{_K} \rightarrow \mathfrak{H}_{_K}$ is an isomorphism and that there exists an open neighbourhood $\mathfrak{U}$ (resp. $\mathfrak{B}$) of the unit section of $\mathfrak{G}$ (resp. $\mathfrak{H}$) such that $j(\mathfrak{U})=\mathfrak{B}$ and that the restriction of $j$ to $\mathfrak{U}$ is an isomorphism of schemes from $\mathfrak{U}$ to $\mathfrak{B}$. Then $j$ is an isomorphism of $\mathfrak{G}$ onto an open subscheme of $\mathfrak{H}$.
 \end{thm}
 
\begin{Cor} \cite[Corollaire 2.2.5]{bt2}  Let us suppose that $\mathfrak{Z}$ and the $\mathfrak{U}_a$ are smooth and that $\mathfrak{G}$ is flat. Then
\begin{enumerate}
\item[(i)] the schemes, $\mathfrak{U}_{\Psi}$, $\mathfrak{Z} \mathfrak{U}_{\Psi}$ (for $\Psi \subset \boldsymbol{\Phi}^{\pm}$ close) and $\mathfrak{G}$ are smooth;
\item[(ii)] the neutral component $\mathfrak{G}^o$ of the group scheme $\mathfrak{G}$ is an open and smooth subgroup scheme of $\mathfrak{G}$;
\item[(iii)] if $A$ is a height $1$ valuation ring or a Dedekind ring, then $\mathfrak{G}^o$ is affine. 
\end{enumerate}
\end{Cor}

\section*{Index of Notations} 
\small
\begin{center}
\begin{tabular}{ll}
    \hline
    \tt{Notation} & \tt{Occurence} \\
    \hline
    $\cO, \kappa, \boldsymbol{\Phi}, \mathfrak{T}, \cA_T$ & \eqref{advr} and \eqref{tubularO} \\
    $ \Gfr_{_f}, \Gfr_{_\theta}$ & \eqref{advr} \\
    $H_{_{\kappa}}$ & \eqref{Hkappa} \\\\
    $\mathcal G^{^{Z_{_{k}}}}_{_{\cO}}$ & \eqref{nerondilat} \\
    $\mathfrak{G}^{^{u}}_{_{{\kappa}}}, \mathfrak{G}^{^{\tt red}}_{_{{\kappa}}}$ & \eqref{levieqn} \\
    $f_e, f_s, \psi, \epsilon$ & \eqref{fefs}, \eqref{psietc}, \eqref{epsilon} \\
    $\boldsymbol{\Phi}_{_G}^{^{\tt u}}$ & \eqref{phigtu} \\
    $\left(\mathfrak{G}^{^{\tt u}}_{_{D}} \right)^{^{\tt 0}}, \mathfrak{U}^{^{\tt 0}}_{_{D,a}}$ & \eqref{HUnipotentExtnD}\\
    $Y^{^{(\i)}}$ & \eqref{basicnots} \\ 
    $\mathfrak{G}^{^{\tt r\'edg\'en}}_{_{Y_{_i}}}$ & \eqref{redgen} \\
    $ \mathfrak{G}^{^{\tt /u}}_{_{Y_{_i}}} $ & \eqref{theunipquot} \\
    $ i_{_{N}}$ & \eqref{in} \\
    $\overline{i_{_N}}$ & \eqref{inbar} \\
    $N_{_{Y}}^{^{sm}}, \mathfrak N_{_Y}$ & \eqref{thesm}, \eqref{Ninc}, \eqref{notachange} \\
\end{tabular}
\end{center}


\begin{thebibliography}{100}
\bibitem[BT1]{bruhattits1} F.Bruhat and J.Tits : {Groupes reductifs sur un corps
local, I},  {\it Publications Math\'ematiques de l'IH\'ES}, {\bf 41} (1972) pp 5-251.


\bibitem[BP24]{bp} V.~Balaji and Y.~Pandey, \emph{On Bruhat-Tits theory over a higher dimensional base}, \'Epijournal de G\'eom\'etrie Alg\'ebrique, Volume 8 (2024), Article No. 14.

\bibitem[BP26]{insa} V.~Balaji and Y.~Pandey, \emph{Bruhat-Tits theory over a higher dimensional base}, Contemporary Research in Mathematics from India, INSA volume, Springer 2026, pp: 321-334.


\bibitem[BS]{base} V.Balaji and C.S. Seshadri, Moduli of parahoric $\mathfrak G$--torsors on a compact Riemann surface, {\em J. Algebraic Geometry}, 24, (2015),1-49.

\bibitem[BLR90]{blr} S.~Bosch, W.~Lutkebohmert and M.~Raynaud, \emph{N\'eron Models}, Ergeb.\ Math.\ Grenzgeb.~(3), vol.~21, Springer-Verlag, Berlin, 1990.




\bibitem[BT84a]{bt2}
  F.~Bruhat, andJ.~Tits, \emph{Groupes r\'eductifs sur un corps local. II. Sch\'emas en groupes. Existence d'une donn\'ee radicielle valu\'ee}, Inst.\ Hautes \'Etudes Sci.\ Publ.\ Math.~\textbf{60} (1984), 197--376.

\bibitem[EGAI]{ega1}
  A.~Grothendieck and J.A. Dieudonn\'e,   \emph{\'Elements de g\'eom\'etrie alg\'ebrique: I. Le langage des sch\'emas},  Inst.\ Hautes \'Etudes Sci.\ Publ.\ Math.~\textbf{24} (1960). 


\bibitem[EGAIV]{ega4}
  A.~Grothendieck and J.A. Dieudonn\'e,   \emph{\'Elements de g\'eom\'etrie alg\'ebrique. IV. \'Etude locale des sch\'emas et des morphismes de sch\'emas},  Inst.\ Hautes \'Etudes Sci.\ Publ.\ Math.~\textbf{24} (1965). 

\bibitem[DG70]{demgab}
  M.~Demazure and P.~Gabriel, 
\emph{Groupes alg\'ebriques. Tome I: G\'eom\'etrie alg\'ebrique, g\'en\'eralit\'es, groupes commutatifs}
(Avec un appendice \emph{Corps de classes local} par M.~Hazewinkel),
Masson \& Cie, \'Editeurs, Paris; North-Holland Publishing Co., Amsterdam, 1970.

\bibitem[IM]{im} N. Iwahori and H. Matsumoto, On some {B}ruhat decomposition and the structure of the
              {H}ecke rings of {$\mathfrak{p}$}-adic {C}hevalley groups, {\it Inst. Hautes \'{E}tudes Sci. Publ. Math.}, 25, 1965, 5--48

\bibitem[MRR20]{mayeux}
  A.~Mayeux, T.~Richarz and M.~Romagny, \emph{N\'eron blowups and low-degree cohomological applications}, Alg.\ Geom.~\textbf{10} (2023), no.~6, 729--753.

\bibitem[McN20]{mcn} G.~McNinch \emph{Reductive subgroup schemes of a parahoric group scheme}, Trans. \ Groups.~\textbf{25} (2020), no.~1, 217--249.

\bibitem[PR]{pr2024} G. Pappas, M. Rapoport, With an appendix by Brian Conrad, On tamely ramified $\mathcal{G}$-bundles on curves, {\em Algebraic Geometry}, 11 (6) (2024), 796-829.

\bibitem[PZ13]{pz}
  G.~Pappas and X.~Zhu, \emph{Local models of Shimura varieties and a conjecture of Kottwitz}, Invent.\ Math.\ \textbf{194} (2013), no.~1, 147--254.

\bibitem[Ray70]{raynaud}
  M.~Raynaud, \emph{Faisceaux amples sur les sch\'emas en groupes et les espaces homog\`enes}, Lecture Notes in Math., vol.~119, Springer-Verlag, Berlin-New York, 1970.

\bibitem[SGA3]{sga3} M.~Demazure and A.~Grothendieck (eds), \emph{Sch\'emas en groupes. I, II, III}, S\'eminaire de G\'eom\'etrie Alg\'ebrique du Bois Marie 1962/64 (SGA 3), Lecture Notes in Math., vols.~151--153, Springer-Verlag, Berlin-New York, 1970.
  


  

\bibitem[Yu15]{yu}
 J.-K.~Yu, \emph{Smooth models associated to concave functions in Bruhat--Tits theory} in: \emph{Autour des sch\'emas en groupes. Vol.~III}, pp.~227–-258, Panor.\ Synth\`eses, vol.~47, Soc.\ Math.\ France, Paris, 2015. 

\end{thebibliography}
\end{document}